\documentclass[11pt,reqno]{amsart}
\usepackage[T1]{fontenc}
\usepackage{graphicx}
\usepackage{cite}

\usepackage{color}
\definecolor{MyLinkColor}{rgb}{0,0,0.4}
\numberwithin{equation}{section}

\newcommand{\id}{\mathop{\rm id}\nolimits}

\newcommand{\re}{\mathop{\rm Re}\nolimits}

\newcommand{\sign}{\mathop{\rm sign}\nolimits}
\newcommand{\PV}{\mathop{\rm PV}\nolimits}

\newcommand{\0}{\Omega}
\newcommand{\e}{\varepsilon}

\newcommand{\p}{\partial}
\newcommand{\wt}{\widetilde}
\newcommand{\ov}{\overline}

\newcommand{\bA}{\mathbb{A}}

\newcommand{\kH}{\mathcal{H}}

\newcommand{\kL}{\mathcal{L}}

\newcommand{\oo}{\ov\omega}

\newcommand{\C}{\mathbb{C}}

\newcommand{\E}{\mathbb{E}}

\newcommand{\R}{\mathbb{R}}
\newcommand{\s}{\mathbb S}
\newcommand{\N}{\mathbb{N}}
\newcommand{\Z}{\mathbb{Z}}
\newcommand{\X}{\mathbb{X}}

\DeclareMathOperator{\supp}{supp}
 
\newtheorem{thm}{Theorem}[section]
\newtheorem{prop}[thm]{Proposition}
\newtheorem{lemma}[thm]{Lemma}

\newtheorem{rem}[thm]{Remark}

\numberwithin{equation}{section}

\title[On a periodic quasilinear Muskat problem]{Well-posedness and stability results for a quasilinear  periodic Muskat problem}

\author[A.--V. Matioc]{Anca--Voichita Matioc}
\address{Institut f\"ur Angewandte Mathematik, Leibniz Universit\"at Hannover, Welfengarten~1, 30167 Hannover, Deutschland.}
\email{matioca@ifam.uni-hannover.de}

\author[B.--V. Matioc]{Bogdan--Vasile Matioc}
\address{Mathematisches Institut, Heinrich-Heine-Universit\"at D\"usseldorf, Universit\"atsstr.~1, 40225 D\"usseldorf, Deutschland.}
\email{Bogdan-Vasile.Matioc@uni-duesseldorf.de}


\usepackage[colorlinks=true,linkcolor=MyLinkColor,citecolor=MyLinkColor]{hyperref} 

\begin{document}

\begin{abstract}
We study the Muskat problem describing the spatially periodic motion of two fluids with equal viscosities under the effect of gravity in  a vertical unbounded two-dimensional geometry.
We first prove that the classical formulation of the problem is equivalent to a nonlocal and nonlinear evolution equation expressed
in terms of singular integrals and  having only the interface between the fluids as unknown.
Secondly, we show that this evolution equation has a  quasilinear structure, which is at a formal level not obvious,  and  we also disclose the parabolic character of the equation.
Exploiting these aspects, we  establish the local well-posedness of the problem for arbitrary initial data in $H^s(\mathbb{S})$, with $s\in(3/2,2)$, 
determine a new criterion for the global existence of solutions,
and uncover a parabolic smoothing property.
Besides, we prove that the zero steady-state solution is exponentially stable.
\end{abstract}

\maketitle

\section{Introduction and the main results}\label{S1}
We study the contour-integral formulation
\begin{equation}\label{P}
\left\{ 
\begin{array}{rlll}
 \p_tf(t,x)\!\!&=&\!\!-\displaystyle\frac{k\Delta_\rho}{4\pi\mu} f'(t,x)  \PV\int_{-\pi}^{\pi}f'(t,x-s)\frac{(T_{[x,s]} f(t))(1+t_{[s]}^2)}{t_{[s]}^2+(T_{[x,s]} f(t))^2 }\, ds\\[2ex]
 &&\!\!\displaystyle-\frac{k\Delta_\rho}{4\pi\mu}  \PV \int_{-\pi}^{\pi}f'(t,x-s)\frac{t_{[s]}[1-(T_{[x,s]} f(t))^2]}{t_{[s]}^2+(T_{[x,s]} f(t))^2 }\, ds,\qquad t>0,\, x\in\R, \\[1ex]
 f(0)\!\!&=&\!\!f_0
\end{array}\right.
\end{equation}
of the Muskat problem, that describes the evolution of the interface $[y=f(t,x)]$ separating two periodic immiscible fluid layers of unbounded height and  with equal viscosity constants  in a 
homogeneous porous medium with permeability 
$k$ (or a vertical Hele-Shaw cell), the fluid system being   close to the rest state far away from the interface, see Section \ref{S2}.\footnote{In fact, the problem \eqref{P} is also a model for the evolution
of a fluid system that moves  vertically  with velocity $(0,V)$, for some $V\in\R$, and for which  the interface between the fluids is parameterized as the graph $[y=f(t,x)+Vt]$, cf. \cite{M17x}.}
The unknown $f=f(t,x)$ is thus assumed to be $2\pi$-periodic with respect to the variable $x$.
For this reason, we use $\s$ to denote the unit circle $\R/2\pi\Z,$ functions depending on $x\in\s $ being $2\pi$-periodic with respect to the real variable $x$. 
We further   denote by $g$ the Earth's gravity, $\rho_\pm$ is the density of the fluid $\pm$ which is located at $\0_\pm(t)$, where
\[ \0_-(t):=[y<f(t,x)]\qquad\text{and}\qquad \0_+(t):=[y>f(t,x)],\]
and $\mu_-=\mu_+=:\mu$ is the viscosity coefficient of the  fluids.
For brevity, we write   $(\,\cdot\,)'$ for the spatial derivative  $\p_x $ and $\PV$ stands for the principle value which is taken at $s=0$.
Our analysis is in the regime where the Rayleigh-Taylor condition
\begin{equation}\label{RTC}
\Delta_\rho:=g(\rho_--\rho_+)>0
\end{equation}
is satisfied.
Furthermore, in order to  keep the notation short we have set
\[
\delta_{[x,s]}f:=f(x)-f(x-s),\qquad T_{[x,s]} f=\tanh\Big(\frac{\delta_{[x,s]}f}{2}\Big), \qquad t_{[s]}=\tan\Big(\frac{s}{2}\Big).
\]
Our analysis covers only the Muskat problem without surface tension, because we expect that  when surface tension is taken into account we obtain the very same  
results as in the general case when the viscosity constants are not necessarily equal (this issue will be addressed in a subsequent paper). Such a behavior is  not expected when neglecting surface tension, cf. \cite{M16x, M17x}. 

Formulated by the American petroleum engineer M. Muskat in 1934 (see \cite{Mu34}) to describe  the intrusion of water into an oil sand, the  Muskat problem is a classical model in the petroleum engineering and as such it has received 
much attention also in the field of applied mathematics.
Because of the complexity of the mathematical formulation most of the  progress in the analysis of this problem is quite recent, see the references  
\cite{Y96, A04, BCG14, BCS16, BS16x, CGSV15x, CCG11, CCG13b, CG07, CG10, CGO14, CGCSS14x, SCH04, BV14, EMM12a,  EMW15, M16x, CCGS13, GB14, CCFG13, CCFGL12, GG14, CGFL11, 
GS14,   PS16, PS16x, A14, FT03, HTY97,  T16}. 
 Moreover, a large body of  literature considers  the special case when the two fluids have equal viscosities and the surface tension effects are neglected, 
 as in this particular context the equations of motion (see~\eqref{PB}) have an elegant contour-integral formulation (see~\eqref{P} in the periodic setting).\smallskip
 
 With respect to the goals of this paper, we mention  the following results pertaining to \eqref{P}:
 \begin{itemize}
  \item[$(i)$] {\bf Local existence and uniqueness:}  In \cite{CG07} the authors establish the local existence and uniqueness of solutions to \eqref{P} for general initial data in $f_0\in H^3(\s)$.   
  This result was improved recently in \cite{BCS16} where, in the more general context of fluids with general (not necessarily equal) viscosities,  
 initial data $f_0\in H^2(\s)$ that are small with respect to some $H^{3/2+\e}(\s)$-norm, where $\e\in(0,1)$ is arbitrarily small, are considered.  
 Besides, it was shown in \cite{BCS16} that   solutions corresponding to small  initial data  in $H^2(\s)$ exist globally.
   \item[$(ii)$] {\bf Exponential stability:} In contrast to the case of confined geometries considered in  \cite{EMM12a, EM11a, EEM09c, FT03, PS16x},  stability  results for the zero solution to \eqref{P} are not available.
   Related to this issue we mention the exponential stability result for the zero solution to the one-phase inhomogeneous Muskat problem in $H^2(\s)$,  established only recently in \cite{BS16x}, and the
   decaying estimates  with respect to  $H^r(\s)$-norms, $r\in[0,2)$,  of solutions to \eqref{P} corresponding to small  initial data  in $H^2(\s)$, see \cite{BCS16}. 
   \item[$(iii)$] {\bf Instantaneous real-analyticity:} In \cite{M16x} it was shown,  in the nonperiodic case, that the solutions to the Muskat problem that start in $H^s(\R)$, $s>3/2$, become instantly real-analytic, and  they
   stay so until they cease to exist.
   A similar result is available in the periodic case only for the one-phase problem, see \cite{BCS16}.
 \end{itemize}

 The first result established in Section \ref{S2} of this  paper is the Proposition \ref{PE}, where we   rigorously prove, in a quite general setting, that  the classical formulation  \eqref{PB} of the Muskat problem  and the
 contour-integral formulation \eqref{P} are equivalent.
Motivated by this equivalence property, we then address  the well-posedness of the evolution problem \eqref{P}  -- that is, the local existence and uniqueness of classical solutions as well as
 the continuous dependence of the solutions on the initial data -- in the Sobolev space $H^s(\s),$ with $s\in(3/2,2)$.
 The well-posedness  is stated in our first main result Theorem~\ref{MT1}, where we also prove a parabolic smoothing property for \eqref{P}. 
 Moreover, we  provide  in Theorem \ref{MT1} a new criterion for the global existence of solutions.
 
 In Theorem \ref{MT2} we then establish the first exponential stability result for the zero solution to the unconfined Muskat problem and show how the  physical properties of the fluids influence 
 the convergence rate of small solutions to \eqref{P} towards the zero solution.

 Our analysis is motivated by the  previous paper  \cite{M16x}  where analogous results to Theorem \ref{MT1} and Proposition \ref{PE} have been established in the nonperiodic case.
 The main ingredients exploited in  \cite{M16x} are:  (I)  the quasilinearity of the nonperiodic version 
 \begin{equation}\label{CIF}
 \p_tf(t,x)= \frac{k\Delta_\rho }{2\pi \mu}\PV\int_\R\frac{y(f'(t,x)-f'(t,x-y))}{ y^2+(f(t,x)-f(t,x-y))^2 }\, dy\qquad \text{for $ t>0$, $x\in\R$},
\end{equation} 
 of \eqref{P}$_1$, a property which is obvious at least at formal level since the integrand in the latter equation is linear with respect to the highest order spatial
derivative of $f$, that is $f'$, and; (II) the parabolic character of the problem when assuming that the 
 Rayleigh-Taylor condition  \eqref{RTC} is satisfied.    
With respect to (I) we point out that the situation changes in the periodic case considered herein as the right-hand side of \eqref{P}$_1$  is not linear with respect to $f'$, and it cannot be modified to become linear in $f'$. 
Despite this, we are still able to formulate \eqref{P} in a suitable functional analytic setting as a quasilinear evolution problem (see Section \ref{S3}).  
 Concerning (II), this aspect has been firstly evinced in a bounded geometry in \cite{EMW15}, and subsequently in \cite{M16x,M17x} for the unconfined Muskat problem for fluids with equal or different viscosities.
 It is shown in Section \ref{S4} that the   parabolicity property is a defining feature also for \eqref{P}.  
 These two  properties are also crucial in this paper as they enable us  to use  theory developed by H. Amann   \cite{Am86, Am88, Am93} (see Theorem \ref{T:A} in the Appendix \ref{S:A}) and 
 A. Lunardi \cite{L91, L95} (the nonlinear principle of linearized stability) in the proofs of our main results, see Sections \ref{S4}-\ref{S5}. 
 
Compared to the nonperiodic case studied in \cite{M16x}, the analysis in this paper is more involved for two clear reasons: firstly because, as mentioned above, the equation $\eqref{P}_1$ is not linear with respect to $f'$; and, secondly
because a deep result from harmonic analysis (see the main theorem in  \cite{TM86}) which was used in an essential way in \cite{M16x} has no correspondent in the periodic setting.\smallskip

 The first main result of this paper is the following theorem.
  \begin{thm}\label{MT1}
Let  $\Delta_\rho>0$ and  $s\in(3/2,2) $  be given.  Then, the following hold:
 \begin{itemize}
 \item[$(i)$] \noindent{\bf\em Well-posedness in $H^s(\s)$:} The problem \eqref{P} possesses for each   $f_0\in H^s(\s)$ a unique maximal solution 
\[ \qquad f:=f(\cdot; f_0)\in {\rm C}([0,T_+(f_0)),H^s(\s))\cap {\rm C}((0,T_+(f_0)), H^2(\s))\cap {\rm C}^1((0,T_+(f_0)), H^1(\s)),\]
with $T_+(f_0)\in(0,\infty]$, and $[(t,f_0)\mapsto f(t;f_0)]$ defines a  semiflow on $H^s(\s)$. \\[-2ex]
  \item[$(ii)$] \noindent{\bf\em  Global existence/blow-up criterion:} If   
\[
\sup_{[0,T_+(f_0))\cap[0,T]}\|f(t)\|_{H^s(\s)}<\infty \qquad\text{for all $T>0$,}
\]
then  $T_+(f_0)=\infty$.\\[-2ex]
  \item[$(iii)$]\noindent{\bf\em  Parabolic smoothing:}
  The mapping  $[(t,x)\mapsto f(t,x)]:(0,T_+(f_0))\times\R\to\R$ is real-analytic. In particular, $f(t)$ is a real-analytic function for all $t\in (0,T_+(f_0))$.
 \end{itemize}
\end{thm}
With respect to Theorem \ref{MT1} we add the following.

\begin{rem}\label{R:-1}
 \begin{itemize}
   \item[$(i)$] We conjecture (in both periodic and  nonperiodic settings) that $s=3/2$ is (similarly as for the Camassa-Holm equation \cite{B06}) the  critical Sobolev index for the local well-posedness of the Muskat problem 
   for fluids with equal viscosities. We expect that proving the ill-posedness of \eqref{P} for $s<3/2$ is more difficult than  for the Camassa-Holm equation, as the proof   in  \cite{B06}  exploits the availability of
   some explicit solutions, which have no correspondent for \eqref{P}. Besides, the structure of the Camassa-Holm equation is simpler as the highest order term is a local one. \\[-2ex]
  \item[$(ii)$] We point out that if $T_+(f_0)<\infty$ for  some $f_0\in H^s(\s),$ $s\in(3/2,2)$, then Theorem \ref{MT1}   implies 
\[
\inf_{r\in(3/2,s)}\sup_{[0,T_+(f_0))}\|f(t)\|_{H^r(\s)}=\infty.
\]
 \end{itemize}
\end{rem}

The second main result of the paper is the following asymptotic stability result.

\begin{thm}[Exponential stability]\label{MT2}
  Let  $\Delta_\rho>0$. Then, the zero solution to \eqref{P} is exponentially stable.
  More precisely, given $\omega\in(0,k\Delta_\rho/2\mu)$, there exist constants $r>0$ and $M>0$,  with the property that if $f_0\in H^2(\s)$ satisfies 
  \[
  \text{$\|f_0\|_{H^{2}(\s)}<r$ \qquad and\qquad $\int_{-\pi}^\pi f_0\, dx=0,$}
  \]
 then $T_+(f_0)=\infty$  and
  \begin{align*}
   \|f(t)\|_{H^2(\s)}+\|\p_t f(t)\|_{H^1(\s)}\leq Me^{-\omega t}\|f_0\|_{H^{2}(\s)}\qquad \text{for all $t\geq0.$}
  \end{align*} 
 \end{thm}

In particular, Theorem \ref{MT2} shows how the exponential stability of the  zero solution is influence by the physical properties of the fluids. 
 Solutions that correspond to small data converge  at a faster rate towards the zero solution  if  the density jump or the permeability  are larger or if the viscosity is smaller.
We point out that due to the fact that the integral mean of the initial data is preserved by the flow, see Section \ref{S5}, only solutions with zero integral mean may converge to the zero steady-state.

\section{The equations of motion and the equivalent formulation}\label{S2}

 In this section we  present  the   equations  governing the evolution of the fluids system and we rigorously prove that the latter are equivalent to the contour-integral formulation \eqref{P}.    
 Since for flows in porous media the conservation of momentum equation can be replaced by Darcy's law, cf. e.g. \cite{Be88},
  in the fluid layers  the dynamic is governed by the following equations 
\begin{subequations}\label{PB}
\begin{equation}\label{eq:S1}
\left\{\begin{array}{rllllll}
{\rm div}\,  v_\pm(t)\!\!&=&\!\!0&\text{in  $\Omega_\pm(t)$} , \\[1ex]
v_\pm(t)\!\!&=&\!\!-\cfrac{k}{\mu}\big(\nabla p_\pm(t)+(0,\rho_\pm g)\big)&\text{in $ \0_\pm(t)$} 
\end{array}
\right.
\end{equation} 
 for $t> 0$, where $v_\pm(t):=(v_\pm^1(t),v_\pm^2(t))$ denotes the velocity vector and $p_\pm(t)$ the pressure of the fluid $\pm.$ 
These equations are supplemented by  the natural  boundary conditions at the free interface
\begin{equation}\label{eq:S2}
\left\{\begin{array}{rllllll}
p_+(t)-p_-(t)\!\!&=&\!\!0&\text{on $ [y=f(t,x)]$}, \\[1ex]
 \langle v_+(t)| \nu(t)\rangle\!\!&=&\!\!  \langle v_-(t)| \nu(t)\rangle &\text{on $[y=f(t,x)]$},
\end{array}
\right.
\end{equation} 
where  $\nu(t) $ is the unit normal at $[y=f(t,x)]$ pointing into $\0_+(t)$ and   $\langle \, \cdot\,|\,\cdot\,\rangle$  the inner product in~$\R^2$.
Furthermore, we impose the  far-field boundary  condition
 \begin{equation}\label{eq:S3}
\begin{array}{llllll}
v_\pm(t,x,y)\to 0 &\text{for  $|y|\to\infty$ and uniformly with respect to $x$}, 
\end{array}
\end{equation} 
which states that  the fluid motion is localized, the  fluids being close to the rest state far way from the free interface $[y=f(t,x)]$.
The motion of this interface  is coupled to that of the fluids through  the kinematic boundary condition
 \begin{equation}\label{eq:S4}
 \p_tf(t)\, =\, \langle v_\pm(t)| (-f'(t),1)\rangle \qquad\text{on   $ [y=f(t,x)].$}
\end{equation}  
As we consider periodic flows, it is also assumed that $f(t),$ $v_\pm(t)$, and $p_\pm(t)$ are $2\pi$-periodic with respect to $x$ for all $t>0$. 
Finally, the interface at time $t=0$ is  given
\begin{equation}\label{eq:S5}
f(0)\, =\, f_0.
\end{equation} 
\end{subequations}

The equations \eqref{PB} are known as the Muskat problem and they determine  entirely the dynamic of the fluid system.
We now show that   the Muskat  problem \eqref{PB} is equivalent to  the system \eqref{P} presented in the introduction.
\begin{prop}[Equivalence result]\label{PE}
Let   $T\in(0,\infty] $ be given.
The following are equivalent:
 \begin{itemize}
 \item[$(i)$] the Muskat problem \eqref{PB} for $f\in  {\rm C}([0,T), L_2(\s)) \cap {\rm C}^1((0,T), L_2(\s))$ and
 \begin{align*}
 f(t)\in  H^2(\s), \quad v_\pm(t)\in {\rm C}(\ov{\0_\pm(t)})\cap {\rm C}^1({\0_\pm(t)}), \quad p_\pm(t)\in {\rm C}^1(\ov{\0_\pm(t)})\cap {\rm C}^2({\0_{\pm}(t)}) 
 \end{align*}
 for all $t\in(0,T)$;\\[-2ex]
\item[$(ii)$] the   problem \eqref{P} for $ f\in  {\rm C}([0,T), L_2(\s)) \cap {\rm C}^1((0,T), L_2(\s))$ and  $f(t)\in  H^2(\s)$ for all $t\in(0,T)$.\\[-1.5ex]
\end{itemize}
\end{prop}
\begin{proof}
 In the arguments that follow the dependence on time is not written explicitly.
Given a test function  $\varphi\in {\rm C}^\infty_0(\R^2)$, it is easy to verify  that the vorticity 
\[\omega:={\rm rot\,} v:=\p_xv^2-\p_yv^1\in\mathcal{D}'(\R^2)\]
associated to  the global velocity field $v:=(v^1,v^2):=v_-{\bf 1}_{[y\leq f(x)]}+v_+{\bf 1}_{[y> f(x)]}$ satisfies
\begin{align*}
 \langle\omega,\varphi\rangle=\int_{\R} \ov\omega(x)\varphi(x,f(x))\, dx,
\end{align*}
where
\begin{equation}
\oo:= -\frac{k\Delta_\rho}{\mu} f'\in H^1(\s).\label{ovom}
\end{equation}
We next prove that the velocity $v$ is determine by the function $f$. 
More precisely, we show that $v$ coincides in $\R^2\setminus[y=f(x)]$ with the velocity field $\wt v$ introduced via
\begin{equation}\label{VF1}
\begin{aligned}
 \wt v^1(x,y)&:=-\frac{1}{4\pi} \int_{\s}\oo(s)\frac{\tanh((y-f(s))/2)\big[1+\tan^2((x-s)/2)\big]}{\tan^2((x-s)/2)+\tanh^2((y-f(s))/2) }\, ds,\\[1ex]
 \wt v^2(x,y)&:=\frac{1}{4\pi} \int_{\s}\oo(s)\frac{\tan ((x-s)/2)\big[1-\tanh^2((y-f(s))/2)\big]}{\tan^2((x-s)/2)+\tanh^2((y-f(s))/2) }\, ds 
\end{aligned}
\end{equation}
for $(x,y)\in  \R^2 \setminus[y=f(x)].$
To this end, we first observe that  $\wt v$ is well-defined away from $[y=f(x)]$ and $2\pi$-periodic with respect to $x$.
Recalling \eqref{ovom}, it is not difficult to see that  
the far-field boundary condition \eqref{eq:S3} is also satisfied by $\wt v$.
Furthermore, $\wt v^i\in {\rm C}^1(\R^2\setminus[y=f(x)])$, $i=1,2$, and, letting $\wt v_\pm:=\wt v|_{\0_\pm}$, it holds that 
\begin{equation}\label{dvrt}
 {\rm div}\,  \wt v_\pm =  {\rm rot}\,  \wt v_\pm =0\qquad\text{in  $\Omega_\pm.$}
\end{equation}
As a further step, we   show that $\wt v_\pm\in {\rm C}(\ov{\Omega_\pm})$ and that the value of $\wt v_\pm$ on the boundary $[y=f(x)]$ is given by the formulas
\begin{equation}\label{PF}
\begin{aligned}
 \wt v_\pm^1(x,f(x))&=-\frac{1}{4\pi} \PV\int_{\s}\oo(s)\frac{\tanh((f(x)-f(s))/2)\big[1+\tan^2((x-s)/2)\big]}{\tan^2((x-s)/2)+\tanh^2((f(x)-f(s))/2) }\, ds\\[1ex]
 &\hspace{0.424cm}\mp\frac{1}{2}\frac{\oo(x)}{1+f'^2(x)},\\[1ex]
 \wt v_\pm^2(x,f(x))&=\frac{1}{4\pi}\PV \int_{\s}\oo(s)\frac{\tan ((x-s)/2)\big[1-\tanh^2((f(x)-f(s))/2)\big]}{\tan^2((x-s)/2)+\tanh^2((f(x)-f(s))/2) }\, ds\mp\frac{1}{2}\frac{\oo(x)f'(x)}{1+f'^2(x)},
\end{aligned}
\end{equation}
where
\[
\Big(\frac{\oo(x)}{1+f'^2(x)},\frac{\oo(x)f'(x)}{1+f'^2(x)}\Big)=\frac{\oo(x)(1,f'(x))}{1+f'^2(x)}
\]
is a vector tangent to the free surface $[y=f(x)].$
Moreover, in \eqref{PF} the principle value is taken at $s=x$ (if we integrate over $(x-\pi,x+\pi)$).
Changing variables,    \eqref{PF} can be compactly written as
\begin{equation}\label{PF1}
\begin{aligned}
 \wt v_\pm (x,f(x))&=\frac{1}{4\pi} \PV\int_{-\pi}^{\pi}\oo(x-s)\frac{\big(-(T_{[x,s]}f) (1+t_{[s]}^2),t_{[s]}[1-(T_{[x,s]}f)^2]\big) }{t_{[s]}^2+(T_{[x,s]}f)^2 }\, ds\\[1ex]
 &\hspace{0.424cm}\mp\frac{1}{2}\frac{\oo(x)(1,f'(x))}{1+f'^2(x)},
\end{aligned}
\end{equation}
and now the $\PV$ is taken at zero.
Using the fact that $\oo\in {\rm C}^{1/2}(\s)$ and $f\in {\rm C}^{3/2}(\s)$, it is matter of direct computation  to see, in view of the inequalities $x\leq \tan x$ for $x\in[0,\pi/2)$ and $\tanh x\leq x$ for $x\geq0$,
that the principal value integrals \eqref{PF}  exist.
In order to prove that the functions defined in \eqref{VF1} extend continuously up to the free surface $[y=f(x)]$, it suffices to show that the limits
 $ \wt v_-(x,f(x))$ and $\wt v_+(x,f(x))$  of $\wt v$ at $(x,f(x)),$ when we approach this point
from below the interface $[y=f(x)] $ or from above, respectively,
 exist and that they are as given in \eqref{PF}.
 To this end, we first note  that 
 \[
 \wt v_\pm(z)=\overline{\frac{1}{4\pi i}\int_\Gamma \frac{g(w)}{\tan((w-z)/2)}\, dw}\qquad\text{for $z=(x,y)\not\in[y=f(x)],$}
 \]
where we integrate over  
\[
\Gamma:=\{(s,f(s))\,:\, x-\pi\leq s<x+\pi\}
\]
and where $g:\Gamma\to\C$ is given by
\[
g(s,f(s)):=g(s):=-\frac{\oo(s)(1-if'(s))}{1+f'^2(s)},\qquad s\in\R.
\]
It holds that   $g\in {\rm C}^{1/2}(\s,\R^2)$  and additionally 
\begin{align*}
 \overline{\wt v_\pm}(z)&=\frac{1}{4\pi i}\int_\Gamma \frac{g(w)}{\tan((w-z)/2)}\, dw\\[1ex]
 &=\frac{1}{4\pi i}\int_\Gamma g(w)\Big[\frac{1}{\tan((w-z)/2)}-\frac{2}{w-z}\Big]\, dw
 +\frac{1}{2\pi i}\int_\Gamma \frac{g(w)}{w-z}\, dw.
\end{align*}
Observing that $|\tan z-z|\leq C |z|^3$ for all $|z|\leq 1$, the first integral is not singular for $z=(x,f(x))$ and the limit $z\to (x,f(x))$ can be performed by using classical arguments.
For the second integral we may use 
 the Plemelj theorem, see e.g. \cite{JKL93}, to pass to the limit $z\to (x,f(x))$ and to conclude in this way that    indeed  $\wt v_\pm(z)\to \wt v_\pm(x,f(x))$,  together with $\wt v_\pm\in {\rm C}(\ov{\0_\pm})$.
Setting $V_\pm:=v_\pm-\wt v_\pm$, the function  $V:=(V^1,V^2):=V_-{\bf 1}_{[y\leq f(x)]}+V_+{\bf 1}_{[y> f(x)]}$ belongs to $  {\rm BC}(\R^2)$. 
Let
\[
\psi_\pm(x,y):=\int^{y}_{f(x)}V_\pm ^1(x,s)\, ds-\int_0^x  \langle V_\pm(s,f(s))| (-f'(s),1)\rangle \, ds\qquad\text{for $(x,y)\in\ov \0_\pm,$}
\]
be the stream function  associated to $V_\pm$. 
It then follows from \eqref{dvrt}   that   $\psi:=\psi_-{\bf 1}_{[y\leq f]}+\psi_+{\bf 1}_{[y>f]}$ belongs to $  {\rm C}^1(\R^2)$  and satisfies $\Delta\psi=0$ in $\mathcal{D}'(\R^2)$.
Hence, $\psi$ is the real part of a holomorphic function $u:\C\to\C$.
Since $u'$ is also holomorphic and $u'=\p_x\psi-i\p_y\psi=-(V^2,V^1)$ is bounded and vanishes for $|y|\to\infty$ it follows that $u'=0$, hence $V=0$.
This proves that $v_\pm=\wt v_\pm$.
The kinematic boundary condition \eqref{eq:S5} together with \eqref{PF1} imply that $f$ solves the   evolution problem \eqref{P}.

Vice versa,  given a solution $f$ to \eqref{P},  we define the velocity fields $v_\pm\in {\rm C}(\ov{\0_\pm})\cap{\rm C}^1({\0_\pm})$  by \eqref{PF} and the pressures 
$p_\pm\in {\rm C}^1(\ov{\0_\pm})\cap{\rm C}^2({\0_\pm})$ according to 
\begin{align*}
 p_\pm(x,y):=c_\pm-\frac{\mu}{k}\int_0^x v_\pm^1(s,\pm d)\, ds-\frac{\mu}{k}\int_{\pm d}^y v_\pm^2(x,s)\, ds-\rho_\pm gy, \qquad (x,y)\in\ov\0_\pm
\end{align*}
for some  $d>\|f\|_ \infty$ and $c_\pm\in\R$.  
We note that $p_\pm$ are $2\pi$-periodic  with respect to the $x$-variable if and only if 
\begin{equation}\label{L:LLL}
\int_0^{2\pi} v_\pm^1(s,\pm d)\, ds=0.
\end{equation}
In order to establish \eqref{L:LLL}, we infer from  \eqref{dvrt} that 
\[
\int_0^{2\pi}(\p_xv_\pm^2-\p_y v_\pm^1)(x,\pm y)\, dx=0 \qquad\text{for all $y>\|f\|_\infty,$ }
\]
and, exploiting the periodicity of $v_\pm^2$, we find constants $C_\pm\in\R$ such that
\[
\int_0^{2\pi} v_\pm^1(x,\pm y)\, dx= C_\pm \qquad\text{for all $y>\|f\|_\infty.$ }
\]
Recalling \eqref{eq:S3}, we conclude that $C_\pm=0,$ and this proves   the $2\pi$-periodicity of $p_\pm$.
Finally,    it is easy to see, in view of \eqref{PF1}, that the constants $c_\pm$ can be chosen in such a way that  the tuple $(f, v_\pm, p_\pm)$  solves \eqref{PB}.
\end{proof}

\section{The Muskat problem as a qausilinear evolution equation}\label{S3}
The goal of this section is to show that the Muskat problem \eqref{P} can indeed be formulated as a quasilinear  evolution problem (see \eqref{AF}).
This is the first stage in the analysis of \eqref{P}.  
As mentioned in the introduction, this essential property of \eqref{P} is not obvious at all, in contrast to the nonperiodic case where it is evident  (at least at a formal level).
The departure from the nonperiodic case is due to  the fact that an  argument from  \cite{M16x} that uses integration by parts  is not applicable to $\eqref{P}_1$ in order to further transform this evolution equation.
The key point in the analysis presented below is to relate  the equation $\eqref{P}_1$ to the contour-integral formulation \eqref{CIF} obtained in the nonperiodic case and to enforce
the integration by parts argument from \cite{M16x}. 
This procedure will give rise to several integral terms of ``lower order''\footnote{As it is clear from Theorem~\ref{MT1}, the Muskat problem \eqref{P} is a first order evolution problem.}.

To begin, we first note that \eqref{P} can be expressed, after rescaling the time appropriately,  in the following compact form
\begin{equation}\label{AF}
 \p_tf = \Phi(f)[f],\quad t>0,\qquad f(0)=f_0,
\end{equation}
where $\Phi(f)$ is a linear operator which is decomposed as
\begin{equation}\label{AFf}
 \Phi(f)=\Phi_0(f)-\Phi_1(f)-\Phi_2(f),
\end{equation}
with
\begin{align}
 \Phi_0(f)[h](x)&:=    \PV\int_{-\pi}^{\pi}\frac{\delta_{[x,s]}h'}{t_{[s]}}\frac{1}{1+\big(T_{[x,s]} f/t_{[s]}\big)^2 }\, ds,\label{PHI0}\\[1ex]
 \Phi_1(f)[h](x)&:=   h'(x)\PV\int_{-\pi}^{\pi}\Big[\frac{f'(x-s)}{t_{[s]}}\frac{T_{[x,s]} f}{t_{[s]}}+\frac{1}{t_{[s]}}  \Big]\frac{1}{1+\big(T_{[x,s]} f/t_{[s]}\big)^2 } \, ds,\label{PHI1}\\[1ex]
 \Phi_2(f)[h](x)&:= h'(x) \int_{-\pi}^{\pi}f'(x-s)\frac{T_{[x,s]} f}{1+\big(T_{[x,s]} f/t_{[s]}\big)^2 }\, ds \nonumber\\[1ex]
 &\hspace{0.424cm}-   \int_{-\pi}^{\pi}h'(x-s)\frac{\big(T_{[x,s]} f/t_{[s]}\big)T_{[x,s]} f}{1+\big(T_{[x,s]} f/t_{[s]}\big)^2 }\, ds.\label{PHI2}
\end{align}
In the following the function $f$ is chosen to be an element of the periodic Sobolev space $H^s(\s)$, $s\in(3/2,2),$ and the appropriate space for $h$ is  $H^2(\s)$. 
However, when dealing with ``lower order'' terms we sometimes choose $h\in H^s(\s).$
We first note  that  $\Phi_2$ encompasses only integral terms  that are not singular and therefore the $\PV$ is not needed.
As we shall see below, cf.  Lemma \ref{L:2}, the integral defined  in \eqref{PHI1} and which appears to be singular  can be decomposed as a sum of three integral terms which are all not singular.
Finally, the integral  in \eqref{PHI0} is singular and quite similar  to the right-hand side of \eqref{CIF}, and this term proves to be the most important one in the analysis.

The  main result of this section is the  regularity property stated in Proposition \ref{L:0}, which shows  that \eqref{AF} is indeed a quasilinear evolution problem.
 \begin{prop}\label{L:0}
 Given $s\in(3/2,2)$, it holds that
\begin{equation}\label{G1}
  \Phi\in {\rm C}^{1-}(H^s(\s),\kL(H^2(\s), H^1(\s))).
 \end{equation}
\end{prop}
\begin{proof}
 The assertion is a direct  consequence of \eqref{AFf} and of the Lemmas  \ref{L:1}-\ref{L:3} below.
\end{proof}

Similarly as in the nonperiodic case \cite{M16x}, it can be shown by using the Lemmas \ref{L:1}-\ref{L:3} and other classical arguments, that the operator $\Phi$ is actually real-analytic, that is
\begin{equation}\label{Gom}
  \Phi\in {\rm C}^{\omega}(H^s(\s),\kL(H^2(\s), H^1(\s)))\qquad\text{for each $s\in(3/2,2)$}.
 \end{equation}
The lengthy details are left to the interested reader.\medskip

To  proceed, we  first study the mapping properties of the operator $\Phi_2$.
 \begin{lemma}\label{L:1}
 Given $s\in(3/2,2)$, it holds that
 \begin{equation}\label{MP1}
  \Phi_2\in {\rm C}^{1-}(H^s(\s),\kL(H^2(\s), H^1(\s))).
 \end{equation}
\end{lemma}
\begin{proof} We decompose $\Phi_2(f)[h]=h'\phi_1(f)-\phi_2(f)[h]$, where, given $f,\,  h\in H^s(\s)$,  we  set
 \begin{align}\label{ph1}
  \phi_1(f)(x)&:=\int_{-\pi}^{\pi}f'(x-s)\frac{T_{[x,s]} f}{1+\big(T_{[x,s]} f/t_{[s]}\big)^2 }\, ds,\\[1ex]
  \phi_2(f)[h](x)&:=\int_{-\pi}^{\pi}h'(x-s)\frac{\big(T_{[x,s]} f/t_{[s]}\big) T_{[x,s]} f }{1+\big(T_{[x,s]} f/t_{[s]}\big)^2 }\, ds.\label{ph2}
 \end{align}
Using classical arguments, it follows  that  $\phi_1(f),$ $\phi_2(f)[h]\in {\rm C}^1(\s)$, $i=1,2$,  with
\begin{align*}
  (\phi_1(f))'(x)&=\frac{f'(x)}{2}\int_{-\pi}^{\pi}f'(x-s)\frac{\big[1-(T_{[x,s]} f)^2\big]\big[1-\big(T_{[x,s]} f/t_{[s]}\big)^2\big]}{\big[1+\big(T_{[x,s]} f/t_{[s]}\big)^2 \big]^2}\, ds,\\[1ex]
&\hspace{0.424cm}+\int_{-\pi}^{\pi}f'(x-s)\frac{(1+t_{[s]}^2)(T_{[x,s]} f)^2}{t_{[s]}^2}\frac{T_{[x,s]} f/t_{[s]}}{\big[1+\big(T_{[x,s]} f/t_{[s]}\big)^2 \big]^2}\, ds,
\end{align*}
respectively
\begin{align*}
  (\phi_2(f)[h])'(x)&=f'(x)\int_{-\pi}^{\pi}h'(x-s)\frac{\big[1-(T_{[x,s]} f)^2\big]\big(T_{[x,s]} f/t_{[s]}\big)}{\big[1+\big(T_{[x,s]} f/t_{[s]}\big)^2 \big]^2}\, ds,\\[1ex]
&\hspace{0.424cm}-\frac{1}{2}\int_{-\pi}^{\pi}h'(x-s)\frac{(1+t_{[s]}^2)(T_{[x,s]} f)^2}{t_{[s]}^2}\frac{1-\big(T_{[x,s]} f/t_{[s]}\big)^2}{\big[1+\big(T_{[x,s]} f/t_{[s]}\big)^2 \big]^2}\, ds.
\end{align*}
The previous formulas imply  that $\Phi_2(f)\in\kL(H^2(\s), H^1(\s))$.
To be more precise,  these formulas actually show that 
\begin{align}
 &\phi_1\in {\rm C}^{1-}(H^s(\s), {\rm C}^1(\s)),\label{r0}\\[1ex]
 &\phi_2\in {\rm C}^{1-}(H^s(\s),\kL(H^s(\s), {\rm C}^1(\s))).\label{r1}
\end{align}
The desired local Lipschitz continuity property \eqref{MP1} is a simple consequence of the \eqref{r0}-\eqref{r1}. 
\end{proof}

 Compared to \eqref{MP1}, we have   taken in \eqref{r1}  the variable $h$ from the larger space $H^s(\s),$ with $s\in(3/2,2)$. 
This enables us  later on to identify the operator $\phi_2(f)$   
   as a  ``lower order'' term  in the decomposition \eqref{AFf} of $\Phi(f)$.
 
We next consider the operator $\Phi_1$.
At first glance, the  integral in \eqref{PHI1} is singular and this could hinder us to write the problem \eqref{P} as a quasilinear evolution problem.
Nevertheless, exploiting the  cancellation in a singular integral whose integrand behaves at first order asymptotically for $s\to0$ in the same way as that in \eqref{PHI1},    cf. \eqref{canc}, 
(we are inspired at this point by the analysis in the nonperiodic case \cite{M16x}),
we show below that this singular integral  defines a function in ${\rm C}^1(\s)$ when requiring  merely that $f\in H^s(\R).$
This is an important observation with respect to our goal of expressing $\eqref{P}_1$  as quasilinear evolution equation.   
\begin{lemma}\label{L:2}
Given $s\in(3/2,2)$ and $f\in H^s(\s)$, we let
\begin{equation}\label{po}
  \phi_3(f)(x):=\PV\int_{-\pi}^{\pi}\Big[\frac{f'(x-s)}{t_{[s]}}\frac{T_{[x,s]} f}{t_{[s]}}+\frac{1}{t_{[s]}}  \Big]\frac{1}{1+\big(T_{[x,s]} f/t_{[s]}\big)^2 } \, ds.
 \end{equation}
Then, it holds that 
 \begin{equation}\label{MP2}
  \phi_3\in {\rm C}^{1-}(H^s(\s),{\rm C}^1(\s))\qquad\text{and}\qquad \Phi_1\in {\rm C}^{1-}(H^s(\s),\kL(H^2(\s), H^1(\s))).
 \end{equation}
\end{lemma}
\begin{proof}
 Since $f\in {\rm C}^{s-1/2}(\s)$, standard arguments show that the limit $\phi_3(f)(x)$ exists  (in the sense of $\PV$) for all $x\in\R$.
We now decompose the function $\phi_3(f)$ as follows.
Using the relations
\begin{equation}
|\tanh(x)-x|\leq |x|^3 ,\, \, x\in\R, \qquad\text{and}\qquad|\tan (x)-x|\leq |x|\tan^2(x),\, \, |x|<\pi/2, \label{aa}
\end{equation}
we may write 
\begin{equation*}
 \phi_3(f)=\phi_{3a}(f)+\phi_{3b}(f)+\frac{1}{2}\phi_{3c}(f),
\end{equation*}
where
\begin{align}
 \phi_{3a}(f)(x)&:=\int_{-\pi}^{\pi}\frac{f'(x-s)}{(t_{[s]})^2+ (T_{[x,s]} f )^2 }\Big[\tanh\Big(\frac{\delta_{[x,s]}f}{2}\Big)-\frac{\delta_{[x,s]}f}{2}\Big] \, ds,\label{pa}\\[1ex]
 \phi_{3b}(f)(x)&:=\int_{-\pi}^{\pi}\frac{1}{(t_{[s]})^2+ (T_{[x,s]} f )^2 }\Big[\tan \Big(\frac{s}{2}\Big)-\frac{s}{2}\Big] \, ds,\label{pb}\\[1ex]
 \phi_{3c}(f)(x)&:=\PV\int_{-\pi}^{\pi}\frac{s+f'(x-s)\delta_{[x,s]}f  }{(t_{[s]})^2+ (T_{[x,s]} f )^2 } \, ds. \nonumber
\end{align}
Using the relation 
\begin{equation}
\PV\int_{-\pi}^{\pi}\frac{s+f'(x-s)(\delta_{[x,s]}f)}{s^2+ (\delta_{[x,s]}f )^2 } \, ds=0,\label{canc}
\end{equation}
see \cite{M16x}, we may express $\phi_{3c}(f)$ as 
\begin{align}
 \phi_{3c}(f)(x)&=\int_{-\pi}^{\pi}[s+f'(x-s)(\delta_{[x,s]}f)]\Big[\frac{1}{(t_{[s]})^2+ (T_{[x,s]} f )^2 }-\frac{4}{s^2+ (  \delta_{[x,s]}f )^2 }\Big] \, ds.\label{pc'}
\end{align}
Due to \eqref{aa}, the principal value is not needed in \eqref{pa}, \eqref{pb}, and \eqref{pc'}.
Furthermore,  \eqref{aa} leads us to the conclusion that $\phi_{3a}(f),$ $\phi_{3b}(f),$ $\phi_{3c}(f) \in {\rm C}^1(\s),$ with
\begin{align*} 
 (\phi_{3a}(f))'(x)&=-f'(x)\int_{-\pi}^{\pi}f'(x-s)\frac{(T_{[x,s]} f)\big[1-(T_{[x,s]} f)^2 \big]}{\big[t_{[s]}^2+(T_{[x,s]} f)^2\big]^2 }\Big[T_{[x,s]} f-\frac{\delta_{[x,s]}f}{2}\Big] \, ds\\[1ex]
 &\hspace{0.424cm}-\int_{-\pi}^{\pi}f'(x-s)\frac{t_{[s]}\big[1+(t_{[s]})^2 \big]}{\big[t_{[s]}^2+(T_{[x,s]} f)^2\big]^2 }\Big[T_{[x,s]} f-\frac{\delta_{[x,s]}f}{2}\Big] \, ds\\[1ex]
 &\hspace{0.424cm}-\frac{f'(x)}{2}\int_{-\pi}^{\pi}f'(x-s)\frac{ (T_{[x,s]} f )^2}{(t_{[s]})^2+ (T_{[x,s]} f )^2 }  \, ds,\\[2ex]
(\phi_{3b}(f))'(x)&=- \int_{-\pi}^{\pi}\delta_{[x,s]}f'\frac{t_{[s]}-(s/2)}{\big[t_{[s]}^2+(T_{[x,s]} f)^2\big]^2 }(T_{[x,s]} f)\big[1-(T_{[x,s]} f)^2\big] \, ds,
\end{align*}
respectively
\begin{align*}
 (\phi_{3c}(f))'(x)
 &= \int_{-\pi}^{\pi}s \,\delta_{[x,s]}f'\Big[\frac{8(\delta_{[x,s]}f)}{\big[s^2+ (  \delta_{[x,s]}f )^2\big]^2 }-\frac{(T_{[x,s]} f)\big[1-(T_{[x,s]} f)^2 \big]}{\big[t_{[s]}^2+(T_{[x,s]} f)^2\big]^2 }\Big]\, ds\\[1ex]
 &\hspace{0.424cm}+f'(x)\int_{-\pi}^{\pi}f'(x-s)\Big[\frac{1}{(t_{[s]})^2+ (T_{[x,s]} f )^2 }-\frac{4}{s^2+ (  \delta_{[x,s]}f )^2 }\Big] \, ds\\[1ex]
 &\hspace{0.424cm}+ \int_{-\pi}^{\pi}f'(x-s)\,\delta_{[x,s]}f \Big[\frac{8s}{\big[s^2+ (  \delta_{[x,s]}f )^2 \big]^2}-\frac{t_{[s]}(1+(t_{[s]})^2) }{\big[(t_{[s]})^2+ (T_{[x,s]} f )^2\big]^2 }   \, ds\\[1ex]
   &\hspace{4.25cm}\quad +\frac{8f'(x)(\delta_{[x,s]}f)}{\big[s^2+ (  \delta_{[x,s]}f )^2 \big]^2}-\frac{ f'(x)(T_{[x,s]} f ) (1+(T_{[x,s]} f )^2)}{\big[(t_{[s]})^2+ (T_{[x,s]} f )^2\big]^2 } \Big]\, ds.
\end{align*}
Noticing that $\Phi_1(f) = \phi_3(f)\p_x$, the local Lipschitz continuity properties stated at \eqref{MP2} are direct consequences of the formulas above.  
\end{proof}

We now turn our attention to the singular integral operator $\Phi_0 (f)$ defined \eqref{PHI0}. 
\begin{lemma}\label{L:3}
It holds that 
 \begin{align}\label{MP4}
   \Phi_0\in {\rm C}^{1-}(H^s(\s),\kL(H^2(\s), H^1(\s))).
 \end{align}
\end{lemma}
\begin{proof}
 It is convenient to write
 \begin{align*}
  \Phi_0=\Phi_{a}-\Phi_{b} +2\Phi_{c},
 \end{align*}
where
\begin{align}
 \Phi_{a}(f)[h]&:=     h'\phi_4(f) ,\label{PHIa}\\[1ex]
 \Phi_{b}(f)[h](x)&:=     \int_{-\pi}^{\pi} h'(x-s)\Big[\frac{1}{t_{[s]}}\frac{1}{1+\big(T_{[x,s]} f/t_{[s]}\big)^2} -\frac{1}{s/2}\frac{1}{1+\big(\delta_{[x,s]}f/s\big)^2}\Big]\, ds,\label{PHIb}\\[1ex]
 \Phi_{c}(f)[h](x)&:= \PV\int_{-\pi}^{\pi}\frac{\delta_{[x,s]}h'}{s}\frac{1}{1+\big(\delta_{[x,s]}f/s\big)^2} \, ds,\label{PHIc}
\end{align}
and
\begin{align}\label{p4}
 \phi_4(f)(x):=\int_{-\pi}^{\pi}\Big[\frac{1}{t_{[s]}}\frac{1}{1+\big(T_{[x,s]} f/t_{[s]}\big)^2} -\frac{1}{s/2}\frac{1}{1+\big(\delta_{[x,s]}f/s\big)^2}\Big]\, ds.
\end{align}
The relation \eqref{aa} shows the principal value is not needed in \eqref{PHIb} and \eqref{p4}.
Similarly as in the Lemmas \ref{L:1}-\ref{L:2}, it follows  that  
\begin{align}\label{p4a}
 &\phi_4\in {\rm C}^{1-}(H^s(\s), {\rm C}^1(\s)),\\[1ex]
& \Phi_{b}\in {\rm C}^{1-}(H^s(\s),\kL(H^s(\s), {\rm C}^1(\s))).\label{p4b}
\end{align}

In order to deal with the operator $\Phi_c$ we introduce for each $h\in H^k(\s)$, $k\in \N$, the function 
$\wt h:=h\varphi$, where $\varphi\in {\rm C}^\infty_0(\R,[0,1])$ is a fixed function chosen such that $\varphi=1$ for $|x|\leq 2\pi$ and $\varphi=0$ for $|x|\geq 4\pi.$
Then, $\wt h\in H^k(\R)$ and there exists $C=C(k)$ such that
\begin{align}\label{LP}
 \text{$\|h\|_{H^k(\s)}\leq \|\wt h\|_{H^k(\R)}\leq C\|h\|_{H^k(\s)}$\qquad for all $h\in H^k(\s).$}
\end{align}
For later use, we notice that \eqref{LP} together with  interpolation property  \eqref{IP} implies that for each $r\geq0$ there exists $C=C(r)$ such that 
 \begin{align}\label{LP1}
 \|\wt h\|_{H^{r}(\R)}\leq C \|h\|_{H^r(\s)} \qquad\text{for all $h\in H^r(\s).$}
 \end{align}
 
It is obvious that for each $f,\, h\in H^s(\s)$ the  function  $\Phi_c(f)[h]$ is  well-defined and it belongs to ${\rm C}(\s)$ (the $\PV$ is actually  not needed).
Given $x\in(-\pi,\pi)$, we may write 
\begin{align}\label{E:RRRF}
 \Phi_{c}(f)[h](x)&= \PV\int_{-\pi}^{\pi}\frac{\delta_{[x,s]}{\wt h}'}{s}\frac{1}{1+\big(\delta_{[x,s]}f/s\big)^2} \, ds=A_1(f)[\wt h](x)- A_2(f)[ h](x),
\end{align}
where
\begin{align*}
 A_1(f)[\wt h](x)&:=\PV\int_{\R}\frac{\delta_{[x,s]}{\wt h}'}{s}\frac{1}{1+\big(\delta_{[x,s]}f/s\big)^2} \, ds,\\[1ex]
 A_2(f)[h](x)&:=\PV\int_{|s|>\pi}\frac{\delta_{[x,s]}{ (\varphi h)}'}{s}\frac{1}{1+\big(\delta_{[x,s]}f/s\big)^2} \, ds.
\end{align*}
The latter formulas make sense for arbitrary $x\in\R$ and the principal value needs to be taken also at infinity.
Our first goal is to prove that  $A_1(f)[\wt h], A_2(f)[ h]\in H^1((-\pi,\pi))$. Having shown this property,  \eqref{E:RRRF} together with  $\Phi_c(f)[h]\in {\rm C}(\s)$ implies that $\Phi_{c}(f)[h]\in H^1(\s)$ and that
$(\Phi_{c}(f)[h])'$ is the periodic extension of  $(A_1(f)[\wt h]-A_2(f)[ h])'$ (see Lemma \ref{L:A2}).

 Arguing as in the proof of \cite[Lemma 3.5]{M16x}, it follows from Lemma \ref{L:A1}, that  
\begin{align}\label{gga}
 A_1\in {\rm C}^{1-}(H^s(\s),\kL(H^2(\R), H^1(\R))), 
\end{align}
with
\begin{equation}\label{FO}
\begin{aligned}
 (A_1(f)[\wt h])'(x)& =\PV\int_{\R}\frac{\delta_{[x,s]}{\wt h}''}{s}\frac{1}{1+\big(\delta_{[x,s]}f/s\big)^2} \, ds\\[1ex]
 &\hspace{0.424cm}-2\PV\int_{\R}\frac{\big[\delta_{[x,s]}f/s\big]\big[(\delta_{[x,s]}f'/s\big]\big[\delta_{[x,s]}\wt h'/s\big]}{\big[1+\big(\delta_{[x,s]}f/s\big)^2\big]^2}\, ds.
\end{aligned}
\end{equation}

We now turn our attention to the operator $A_2$, which we decomposed as the difference
\[
 A_2 =A_{2a} -A_{2b},
\]
where, given $x\in(-\pi,\pi)$, we have set
\begin{align*}
A_{2a}(f)[h](x)&:=h'(x) \PV\int_{|s|>\pi}\frac{1}{s}\frac{1}{1+\big(\delta_{[x,s]}f/s\big)^2} \, ds,\\[1ex]
A_{2b}(f)[h](x)&:=\int_{|s|>\pi}\frac{(\varphi h)'(x-s)}{s}\frac{1}{1+\big(\delta_{[x,s]}f/s\big)^2} \, ds)s.
\end{align*}
It can be easily verified that  the function 
\begin{equation}\label{p5}
\begin{aligned}
 \phi_5(f)(x)&:=\PV\int_{|s|>\pi}\frac{1}{s}\frac{1}{1+\big(\delta_{[x,s]}f/s\big)^2} \, ds\\[1ex]
 &\phantom{:}=\int_\pi^\infty\frac{1}{s^3}\frac{(\delta_{[x,-s]}f)^2-(\delta_{[x,s]}f)^2}{[1+\big(\delta_{[x,s]}f/s\big)^2][1+\big(\delta_{[x,-s]}f/s\big)^2]} \, ds 
\end{aligned}
\end{equation}
satisfies 
\begin{align}\label{r5}
 \phi_5\in {\rm C}^{1-}(H^s(\s), {\rm BC}^1(\R)),
\end{align}
and therewith
\[
A_{2a}\in {\rm C}^{1-}(H^s(\s), \kL(H^2(\s), H^1((-\pi,\pi))).
\]
Concerning $A_{2b}$, we recall that  $\supp\varphi\subset[-4\pi,4\pi]$, meaning that  
\begin{align}\label{gg2a}
A_{2b}(f)[h](x)&=\int\limits_{\pi<|s|<5\pi}\frac{(\varphi h)'(x-s)}{s}\frac{1}{1+\big(\delta_{[x,s]}f/s\big)^2} \, ds\qquad\text{for $x\in(-\pi,\pi)$}.
\end{align}
In view of \eqref{gg2a}, we find  that the function $A_{2b}(f)[h]$ is differentiable on $(-\pi,\pi) $ for each $h\in H^s(\s),$ with
\begin{equation}\label{FOO}
\begin{aligned}
 (A_{2b}(f)[h])'(x)&= -2 f'(x)\int\limits_{\pi<|s|<5\pi}\frac{(\varphi h)'(x-s)}{s^2}\frac{ \delta_{[x,s]}f/s }{\big[1+\big(\delta_{[x,s]}f/s\big)^2\big]^2} \, ds\\[1ex]
 &\hspace{0.424cm} -\int\limits_{\pi<|s|<5\pi}\frac{(\varphi h)'(x-s)}{s^2}\frac{ 1-\big(\delta_{[x,s]}f/s\big)^2 }{\big[1+\big(\delta_{[x,s]}f/s\big)^2\big]^2} \, ds,
\end{aligned}
\end{equation}
and therewith we find
\begin{align}\label{ggb}
 A_{2b}\in {\rm C}^{1-}(H^s(\s), \kL(H^s(\s), {\rm C}^1([-\pi,\pi]))).
\end{align}
The desired claim \eqref{MP4} is a direct consequence of \eqref{p4a}, \eqref{p4b}, \eqref{E:RRRF}, \eqref{gga}, \eqref{gg2a}, and \eqref{ggb}. 
\end{proof}

\section{The sectoriality property of the principal part}\label{S4}
The first goal of this section is to prove   that the  operator $\Phi(f)$,
regarded as an unbounded operator in $H^1(\s)$ with definition domain $H^2(\s)$ is, for each $f\in H^s(\s)$, $s\in(3/2,2)$,
the generator of a strongly continuous  analytic semigroup in $\kL(H^1(\s))$, that is (in the notation used in \cite{Am95})
\begin{align*} 
  -\Phi(f)\in \kH(H^2(\s), H^1(\s)).
 \end{align*}
To this end it suffices  to show that  the complexification of this unbounded  operator, which we denote again by $\Phi(f)$, that is the operator
 \[  \big[h=u+iv\mapsto \Phi(f)[u]+i\Phi(f)[v]\big]:H^2(\s,\C)\subset H^1(\s,\C)\to H^1(\s,\C),\]
 generates such a  semigroup in $\kL(H^1(\s,\C))$, see \cite[Corollary 2.1.3]{L95}\footnote{The variable $h$ is in the following complex valued, while $f$ is arbitrary (but fixed) and real-valued. 
Having made this convention, we use  $H^s(\s)$ to denote both Sobolev spaces of real- or complex-valued functions. 
This applies in the entire section, excepting the proof of Theorem \ref{MT1} where we come back to the setting of 
real-valued Sobolev functions considered in all of the other sections.}. 

\noindent The desired generator property is established in Theorem \ref{T:1}.
In fact,    we only need to prove that there  exist constants $\omega>0$ and $\kappa\geq1$ such that 
 \begin{align}\label{13K}
&\omega- \Phi(f)\in{\rm Isom}(H^2(\s), H^1(\s))
\end{align}
and
 \begin{align}
\label{14K}
& \kappa\|(\lambda-\Phi(f))[h]\|_{H^1(\s)}\geq  |\lambda|\cdot\|h\|_{H^1(\s)}+\|h\|_{H^2(\s)}
\end{align}
for all  $\lambda\in\C$ with $\re \lambda\geq \omega$  and $h\in H^2(\s)$, cf. \cite{Am95}.
We end  the section by presenting the proof of our first main result  Theorem \ref{MT1}.

To start, we choose for each  integer $p\geq 3$  a set   $\{\pi_j^p\,:\,{1\leq j\leq 2^{p+1}}\}\subset  {\rm C}^\infty(\s,[0,1])$ such that
\begin{align*}
\bullet\,\,\,\, \,\, & \text{$\supp \pi_j^p=\cup_{n\in\Z} \big(2\pi n+ I_j^p\big)$ and $I_j^p:=[j-5/3,j-1/3] \frac{\pi}{2^p};$}\\[1ex]
 \bullet\,\,\,\, \,\, & \text{ $\sum_{j=1}^{2^{p+1}}\pi_j^p=1 $ in ${\rm C}(\s)$.}
\end{align*}
We call  $\{\pi_j^p\,:\,{1\leq j\leq 2^{p+1}}\}$  a {\em $p$-partition of unity}.
Moreover, let $\{\chi_j^p\,:\,{1\leq j\leq 2^{p+1}}\}\subset  {\rm C}^\infty(\s,[0,1])$ be an associated set of functions such that 
\begin{align}
\bullet\,\,\,\, \,\, & \text{$\supp \chi_j^p=\cup_{n\in\Z} \big(2\pi n+ J_j^p\big)$ with  $I_j^p\subset J_j^p:=[j-8/3,j+2/3] \frac{\pi}{2^p}$;}\label{JJJ0}\\[1ex]
 \bullet\,\,\,\, \,\, & \text{ $\chi_j^p=1$ on $I_j^p$.}\label{JJJ1}
\end{align}
The following remark is a simple exercise.

\begin{rem}\label{R:1}
 Let $k, p\in\N$ with $ p\geq 3$ be given and let $\{\pi_j^p\,:\,{1\leq j\leq 2^{p+1}}\}$ be a  $p$-partition of unity.
 The mapping
$$\big[h\mapsto \max_{1\leq j\leq 2^{p+1}} \|\pi_j^p h\|_{H^k(\s)}\big]: H^k(\s)\to\R$$
defines a norm on $H^{k}(\s)$ which is  equivalent to the standard  Sobolev norm. 
\end{rem}

With respect to our goal \eqref{13K}, it is convenient to introduce the continuous path 
\[[\tau\mapsto \Phi(\tau f)]:[0,1]\to \kL(H^2(\s), H^1(\s)),\]
which transforms the operator $\Phi(f)$ into the well-known  operator 
\begin{align*}
 \Phi(0)[h](x)=-\PV\int_{-\pi}^\pi\frac{h'(x-s)}{t_{[s]}}\, ds=-2\pi H[h'](x),
\end{align*}
where $H$ denotes the periodic Hilbert transform, see e.g. \cite{T04}.
Since $H$ is the Fourier multiplier with symbol $(-i\sign(m))_{m\in\Z},$ it follows that $\Phi(0)=-2\pi(-\p_x^2)^{1/2}$
is the Fourier multiplier with symbol $(-2\pi |m|)_{m\in\Z}.$
The following theorem is a commutator type result which states  that $\Phi(\tau f)$ can be locally approximated by Fourier multipliers  that can be explicitely determined.

\begin{thm}\label{T1} 
Let  $f\in H^s(\s)$, $s\in(3/2,2),$  and     $\mu>0$ be given, and set
\[
s':=\max\Big\{s,\frac{11-2s}{4}\Big\}\in(3/2,2).
\]

Then, there exist $p\geq3$, a finite $p$-partition of unity  $\{\pi_j^p\,:\, 1\leq j\leq 2^{p+1}\} $, a constant $K=K(p)$, and for each  $ j\in\{1,\ldots,2^{p+1}\}$ and $\tau\in[0,1]$ there 
exist   operators $$\bA_{j,\tau}\in\kL(H^2(\s), H^{1}(\s))$$
 such that 
 \begin{equation}\label{DE3}
  \|\pi_j^p\Phi(\tau f)[h]-\bA_{j,\tau}[\pi^p_j h]\|_{H^1(\s)}\leq \mu \|\pi_j^p h\|_{H^2(\s)}+K\|  h\|_{H^{s'}(\s)}
 \end{equation}
 for all $ j\in\{1,\ldots, 2^{p+1}\}$, $\tau\in[0,1],$ and  $h\in H^2(\s)$. 
 The operators $\bA_{j,\tau}$ are defined  by 
  \begin{align} 
 \bA_{j,\tau }:=&a_\tau(x_j^p)\p_x-b_\tau(x_j^p)(-\p_x^2)^{1/2},\label{FM1}
 \end{align}
 where $x_j^p\in I_j^p$ is arbitrary, but fixed. Furthermore,
 \begin{align}\label{at}
  a_\tau:=\phi_4(\tau f)+2\phi_6(\tau f)-\phi_1(\tau f)-\phi_3(\tau f)\qquad\text{and}\qquad b_\tau:=\frac{2\pi}{1+\tau^2f'^2},              
 \end{align}
 where $\phi_1$, $\phi_3,$ and $\phi_4$ are defined in \eqref{ph1}, \eqref{po}, and \eqref{p4}, respectively, and  with 
  \begin{align*}
\phi_6(f)(x) :=&\PV\int_{-\pi}^\pi \frac{1}{s}\frac{1}{1+(\delta_{[x,s]}f/s )^2}\, ds.
\end{align*}
\end{thm}
\begin{proof} 
 Let   $\{\pi_j^p\,:\, 1\leq j\leq 2^{p+1}\}$, with $p\geq 3$ to be fixed later on, be a $p$-partition of unity 
 and  let $\{\chi_j^p\,:\, 1\leq j\leq 2^{p+1}\} $ be a family associated to this $p$-partition of unity which satisfies \eqref{JJJ0}-\eqref{JJJ1}.
 In the following we let $C$  denote constants which are
independent of $p\in\N$, $h\in H^2(\s)$, $\tau\in [0,1]$, and $j \in \{1, \ldots, 2^{p+1}\}$, and the constants  denoted by $K$ may depend only upon $p.$\medskip

\noindent{\em Step 1: Some ``lower order'' terms.\,\,} In view of \eqref{AFf}, it follows directly from \eqref{PHI0}-\eqref{PHI2} and Lemma \ref{L:2}, that 
\begin{align*}
 \|\Phi(\tau f)[h]\|_{\infty}\leq C\|h\|_{H^{s}(\s)},
\end{align*}
and therewith
\begin{align*} 
 \|\pi_j^p\Phi(\tau f)[h]\|_{L_2(\s)}\leq C\|h\|_{H^{s}(\s)}.
\end{align*}
Moreover, combining \eqref{r0}, \eqref{MP2}, \eqref{p4a}, and \eqref{UET1}, 
it holds
\[
\sup_{\tau\in[0,1] }\|a_\tau\|_\infty\leq C,
\]
and therewith we get
\begin{equation*}
  \|\bA_{j,\tau}[\pi^p_j h]\|_{L_2(\s)}\leq  K\|  h\|_{H^{1}(\s)}.
 \end{equation*}
 These  estimates lead us to
\begin{equation}\label{E1}
\begin{aligned}
 & \|\pi_j^p\Phi(\tau f)[h]-\bA_{j,\tau}[\pi^p_j h]\|_{L_2(\s)}\leq \|\Phi(\tau f)[h]\|_{L_2(\s)}+\|\bA_{j,\tau}[\pi^p_j h]\|_{2}\leq K\|  h\|_{H^{s}(\s)},\\[1ex]
 &   \|(\pi_j^p)'\Phi(\tau f)[h]\|_{L_2(\s)}\leq   K\|  h\|_{H^{s}(\s)},
  \end{aligned}
 \end{equation}
and we are   left to estimate the quantity
$\|\pi_j^p(\Phi(\tau f)[h])'-(\bA_{j,\tau}[\pi^p_j h])'\|_{L_2(\s)}$.
To this end we use  the decomposition  $\Phi(f)=\Phi_0(f)-\Phi_1(f)-\Phi_2(f)$ provided in  \eqref{AFf} and we establish suitable estimates for each of these three operators separately. \medskip

\noindent{\em Step 2: The operator $\Phi_1(f).$\,\,}
Letting
\[
\bA_{j,\tau}^1:=\phi_3(\tau f)(x_j^p)\p_x,
\]
 we deduce from Lemma \ref{L:2}   and \eqref{JJJ0}-\eqref{JJJ1} that 
\begin{equation}\label{E2}
 \begin{aligned}
 &\hspace{-1cm}\|\pi_j^p(\Phi_1(\tau f)[h])'-(\bA_{j,\tau}^1[\pi_j^p h])'\|_{L_2(\s)}\\[1ex]
 &\leq \|(\phi_3(\tau f)-\phi_3(\tau f)(x_j^p))(\pi_j^p h)''\|_{L_2(\s)}+K\|h\|_{H^1(\s)}\\[1ex]
 &\leq \|(\phi_3(\tau f)-\phi_3(\tau f)(x_j^p))\chi_j^p\|_{\infty}\|\pi_j^p h\|_{H^2(\s)}+K\|h\|_{H^1(\s)}\\[1ex]
 &\leq \frac{\mu}{3}\|\pi_j^p h\|_{H^2(\s)}+K\|h\|_{H^1(\s)}
\end{aligned}
\end{equation}
if $p$ is chosen sufficiently large.\medskip

\noindent{\em Step 3: The operator $\Phi_2(f).$\,\,} Letting 
\[
\bA_{j,\tau}^2:=\phi_1(\tau f)(x_j^p)\p_x,
\]
 we deduce from  \eqref{r0} and  \eqref{r1}, similarly as  above, that
 \begin{equation}\label{E3}
\begin{aligned}
 &\hspace{-1cm}\|\pi_j^p(\Phi_2(\tau f)[h])'-(\bA_{j,\tau}^2[\pi_j^p h])'\|_{L_2(\s)}\\[1ex]
 &\leq\|\pi_j^p(\phi_1(\tau f)h')'-(\bA_{j,\tau}^2[\pi_j^p h])'\|_{L_2(\s)}+\|\pi_j^p(\phi_2(\tau f)[h])'\|_{L_2(\s)}\\[1ex]
 &\leq \|(\phi_1(\tau f)-\phi_1(\tau f)(x_j^p))\chi_j^p\|_{\infty}\|\pi_j^p h\|_{H^2(\s)}+K\|h\|_{H^s(\s)} \\[1ex]
 &\leq \frac{\mu}{3}\|\pi_j^p h\|_{H^2(\s)}+K\|h\|_{H^s(\s)}
\end{aligned}
\end{equation}
if $p$ is again   sufficiently large.\medskip

\noindent{\em Step 4: The operator $\Phi_0(f).$\,\,} The estimates for this operator  are more involved than for the other two.
To start, we introduce
\begin{align*}
 \bA_{j,\tau}^0 &:= [\phi_4(\tau f)(x_j^p)+2\phi_6(\tau f)(x_j^p)]\p_x +\frac{2}{1+\tau ^2f'^2(x_j^p)}\Phi_c(0)[h],
\end{align*}
  where  $\Phi_c$ is defined in   \eqref{PHIc}. 
 Recalling Lemma \ref{L:3}, it holds that  
\begin{align*}
 (\Phi_0(\tau f)[h])' &=(h'\phi_4(\tau f))'-(\Phi_b(\tau f)[h])'+(2\Phi_c(\tau f)[h])'.
\end{align*}
Together with  \eqref{p4a} and \eqref{p4b} we obtain, similarly as in the previous 2 steps,   for $p$ sufficiently large, that
\begin{equation}\label{ll1}
\begin{aligned}
&\hspace{-1cm} \|\pi_j^p(\Phi_0(\tau f)[h])'-(\bA_{j,\tau}^0[\pi_j^p h])'\|_{L_2(\s)}  \\[1ex]
&\leq   \| \pi_j^p (h'\phi_4(\tau f))' - \phi_4(\tau f)(x_j^p)(\pi_j^p h)''\|_{L_2(\s)}\\[1ex]
&\hspace{0.424cm}+\| (\Phi_b(\tau f)[h])'\|_{L_2(\s)}\\[1ex]
 & \hspace{0.424cm}+2\Big\| \pi_j^p(\Phi_c(\tau f)[h])'-\phi_6(\tau f)(x_j^p)(\pi_j^p h)'' -\frac{1}{1+\tau^2f'^2(x_j^p)} (\Phi_c(0)[\pi_j^ph])'\Big\|_{L_2(\s)} \\[1ex]
 &\leq 2\Big\| \pi_j^p(\Phi_c(\tau f)[h])'-\phi_6(\tau f)(x_j^p)(\pi_j^p h)'' -\frac{1}{1+\tau^2f'^2(x_j^p)} (\Phi_c(0)[\pi_j^ph])'\Big\|_{L_2((-\pi,\pi))}  \\[1ex]
  & \hspace{0.424cm}+ \frac{\mu}{9}\|\pi_j^ph\|_{H^2(\s)}+ K\|h\|_{H^s(\s)}. 
\end{aligned}
\end{equation}

 In order to estimate the remaining term in \eqref{ll1}, we recall from Lemma \ref{L:3} that in $(-\pi,\pi) $  the derivative $(\Phi_c(\tau f)[h])'$ can be represented as follows
 \begin{align*}
  (\Phi_c(\tau f)[h])'&=\PV\int_\R\frac{\delta_{[\cdot, s]}(\varphi h)''}{s}\frac{1}{1+\big(\tau\delta_{[\cdot,s]}f/s\big)^2} \, ds\\[1ex] 
  &\hspace{0.424cm} -2\tau^2\PV\int_{\R}\frac{(\delta_{[\cdot,s]}f/s\big)(\delta_{[\cdot,s]}f'/s\big)(\delta_{[\cdot,s]}(\varphi h)'/s\big)}{\big[1+\big(\tau\delta_{[\cdot,s]}f/s\big)^2\big]^2}\, ds \\[1ex]
  &\hspace{0.424cm}-( h '\phi_5(\tau f))'+(A_{2b}(\tau f)[h])',
 \end{align*}
 where $\phi_5$  is defined in \eqref{p5},  $\varphi\in {\rm C}^\infty_0(\R,[0,1])$ is the function chosen in the proof of Lemma \ref{L:3}, and $(A_{2b}(\tau f)[h])'$ is given  in \eqref{FOO}.
 Using Lemma \ref{L:A1} $(iii)$ (with $r=s$ and $\tau=7/4-s/2$) together with \eqref{LP1} and the relation \eqref{ggb}, we get
 \begin{equation}\label{ll2}
 \begin{aligned}
  &\hspace{-1cm}\Big\|\PV\int_{\R}\frac{(\delta_{[\cdot,s]}f/s\big)(\delta_{[\cdot,s]}f'/s\big)(\delta_{[\cdot,s]}(\varphi h)'/s\big)}{\big[1+\big(\tau\delta_{[\cdot,s]}f/s\big)^2\big]^2}\, ds\Big\|_{L_2((-\pi,\pi))}
  +\|(A_{2b}(\tau f)[h])'\|_{L_2((-\pi,\pi))}\\[1ex]
  &\leq C\|h\|_{H^{s'}(\s)}.
 \end{aligned}
 \end{equation}

 Hence, we infer from    \eqref{r5}, \eqref{ll1}, \eqref{ll2}, and \eqref{UET1} that 
\begin{equation}\label{Chopin}
 \begin{aligned}
\|\pi_j^p(\Phi_0(\tau f)[h])'-(\bA_{j,\tau}^0[\pi_j^p h])'\|_{L_2(\s)}&\leq  \|T_1[h]\|_{L_2(\s)}+2\|T_2[h]\|_{L_2((-\pi,\pi))}\\[1ex]
&\hspace{0.424cm}+ \frac{\mu}{9}\|\pi_j^ph\|_{H^2(\s)}+K\|h\|_{H^{s'}(\s)},
\end{aligned}
\end{equation}
where
\begin{align*}
 T_1[h]&:=(\pi_j^ph)''(\phi_6(\tau f)-\phi_6(\tau f)(x_j^p)), \\[1ex]
 T_2[h]&:=\pi_j^p\PV\int_\R\frac{ (\varphi h)''(\cdot-s)}{s}\frac{1}{1+\big(\tau\delta_{[\cdot,s]}f/s\big)^2} \, ds
 -\frac{1}{1+\tau^2f'^2(x_j^p)} \PV\int_\R\frac{ (\pi_j^p\varphi h)''(\cdot-s)}{s}  \, ds. 
\end{align*}
With regard to $T_1[h]$, it follows from \eqref{UET1} that if $p$ is sufficiently large, then
\begin{align}
 \|T_1[h]\|_{L_2(\s)}\leq \frac{\mu}{9}\|\pi_j^p h\|_{H^2(\s)}.\label{E4}
\end{align}

We now consider the term $T_2[h]$.
Using the notation introduced in Lemma \ref{L:A1} $(i)$, we have
\begin{align*}
\pi_j^p\PV\int_\R\frac{ (\varphi h)''(\cdot-s)}{s}\frac{1}{1+\big(\tau\delta_{[\cdot,s]}f/s\big)^2} \, ds&=B_{0,1}(\tau f)[\pi_j^p(\varphi h)'']\\[1ex]
&\hspace{0.424cm}+ \int_\R (\varphi h)''(\cdot-s) \frac{ \delta_{[\cdot,s]}\pi_j^p/s}{1+\big(\tau\delta_{[\cdot,s]}f/s\big)^2} \, ds,
\end{align*}
and integrating the last term by parts we are led to
\begin{align*}
 T_2[h]&= -B_{1,1}(\tau f)[\pi_j^p,(\varphi h)']-2\tau^2 B_{2,2}(\tau f, \tau f)[f,\pi_j^p, f'(\varphi h)']\\[1ex]
 &\hspace{0,424cm}+2\tau^2 B_{3,2}(\tau f, \tau f)[f,f,\pi_j^p, (\varphi h)']-B_{0,1}(\tau f)[(\pi_j^p)''\varphi h+ (\pi_j^p)'(\varphi h)']\\[1ex]
 &\hspace{0,424cm}+B_{0,1}(\tau f)[(\pi_j^p\varphi h)'']-B_{0,1}(\tau f'(x_j^p){\rm\id}_\R)[(\pi_j^p\varphi h)''].
\end{align*}
  Lemma  \ref{L:A1} $(i)$ together with \eqref{LP1} yields
\begin{equation}\label{E5}
\begin{aligned}
 \|T_2[h]\|_{L_2((-\pi,\pi))}&\leq \|B_{0,1}(\tau f)[(\pi_j^p\varphi h)'']-B_{0,1}(\tau f'(x_j^p){\rm\id}_\R)[(\pi_j^p\varphi h)'']\|_{L_2((-\pi,\pi))} \\[1ex]
 &\hspace{0,424cm}+K\| h\|_{H^1(\s)}. 
\end{aligned}
\end{equation}
Furthermore, an algebraic computation shows that
\begin{align*}
 &\hspace{-1cm}B_{0,1}(\tau f)[(\pi_j^p\varphi h)'']-B_{0,1}(\tau f'(x_j^p){\rm\id}_\R)[(\pi_j^p\varphi h)''] \\[1ex]
 &=-\tau ^2 B_{2,2}(\tau f, \tau f'(x_j^p){\rm\id}_\R)[f- f'(x_j^p){\rm\id}_\R,f+f'(x_j^p){\rm\id}_\R )(\pi_j^p\varphi h)''],
\end{align*}
and, recalling \eqref{JJJ1}, we may write 
\begin{align*}
  B_{0,1}(\tau f)[(\pi_j^p\varphi h)'']-B_{0,1}(\tau f'(x_j^p){\rm\id}_\R)[(\pi_j^p\varphi h)'']=\tau^2(T_{2a}[h]-T_{2b}[h]),
\end{align*}
where
\begin{align*} 
&T_{2a}[h]:=\int_\R \frac{\big[\delta_{[\cdot,s]}(f-f'(x_j^p){\rm id}_{\R})/s\big]\big[\delta_{[\cdot,s]}(f+f'(x_j^p){\rm id}_{\R})/s\big]\big(\delta_{[\cdot,s]}\chi_j^p/s\big)}
  {\big[1+\tau^2f'^2(x_j^p)\big]\big[1+\tau^2\big( \delta_{[\cdot,s]} f/s\big)^2\big]} (\pi_j^p \varphi h)''(\cdot-s) \,ds,\\[1ex]
 &T_{2b}[h]:=\chi_j^p B_{2,2}(\tau f, \tau f'(x_j^p){\rm\id}_\R)[f- f'(x_j^p){\rm\id}_\R,f+f'(x_j^p){\rm\id}_\R ),(\pi_j^p\varphi h)''].
\end{align*}
Using integration by parts  we   get, similarly as in the deduction of \eqref{E5}, that 
\begin{align}
 \|T_{2a}[h]\|_{L_2((-\pi,\pi))}\leq K\| h\|_{H^1(\s)},\label{E6}
\end{align}
and we are  left to estimate $\|T_{2b}[h]\|_{L_2((-\pi,\pi))}.$
Since $\supp \varphi\subset[-4\pi,4\pi]$,  it holds
\[
\|T_{2b}[h]\|_{L_2((-\pi,\pi))}\leq\|T_{2c}[h]\|_{L_2((-\pi,\pi))}+\|T_{2d}[h]\|_{L_2((-\pi,\pi))},
\]
where
\begin{align*}
 T_{2c}[h]&:={\bf 1}_{(-\pi,\pi)}\chi_j^p \PV\int\limits_{|s|<\frac{\pi}{2^p}} \frac{\big[\delta_{[\cdot,s]}(f-f'(x_j^p){\rm id}_{\R})/s\big]\big[\delta_{[\cdot,s]}(f+f'(x_j^p){\rm id}_{\R})/s\big]}
  {\big[1+\tau^2f'^2(x_j^p)\big]\big[1+\tau^2\big( \delta_{[\cdot,s]} f/s\big)^2\big]} \frac{(\pi_j^p \varphi h)''(\cdot-s)}{s} \,ds,\\[1ex]
   T_{2d}[h]&:=\chi_j^p  \int\limits_{\frac{\pi}{2^p}<|s|<5\pi} \frac{\big[\delta_{[\cdot,s]}(f-f'(x_j^p){\rm id}_{\R})/s\big]\big[\delta_{[\cdot,s]}(f+f'(x_j^p){\rm id}_{\R})/s\big]}
  {\big[1+\tau^2f'^2(x_j^p)\big]\big[1+\tau^2\big( \delta_{[\cdot,s]} f/s\big)^2\big]} \frac{(\pi_j^p \varphi h)''(\cdot-s)}{s} \,ds.
\end{align*}
Integration by parts obviously yields
\begin{align}
 \|T_{2d}[h]\|_{L_2((-\pi,\pi))}&\leq K\| h\|_{H^s(\s)}.\label{E7}
\end{align}
To deal with $T_{2c}[h],$ we notice that if $T_{2c}[h](x)\neq0,$ then necessarily 
$$x\in(-\pi,\pi)\cap \Big(\cup_{n\in\Z} \big(2\pi n+ J_j^p\big)\Big).$$
We distinguish two cases.
\begin{itemize}
 \item[$(i)$] For $j\not\in\{2^p, 2^{p}+1, 2^{p}+2\}$, it holds that  $(-\pi,\pi)\cap  \big(2\pi n+ J_j^p\big)\neq\emptyset$ for a single value $n_j^p,$ which is either $-1$ or $0,$ and 
\[
(-\pi,\pi)\cap  \big(2\pi n_j^p+ J_j^p\big)=2\pi n_j^p+ J_j^p=:[a_j^p,b_j^p].
\]
\item[$(ii)$] For $j \in\{2^p, 2^{p}+1, 2^{p}+2\}$, it holds that  $(-\pi,\pi)\cap  \big(2\pi n+ J_j^p\big)\neq\emptyset$ if and only if $n_j^p\in\{-1,0\}$ and 
\[
(-\pi,\pi)\cap   J_j^p =:[b_j^p,\pi), \qquad (-\pi,\pi)\cap  \big(-2\pi+ J_j^p\big)=:(-\pi,a_j^p].
\]
\end{itemize}
The values of $a_j^p,$ $b_j^p$ can be computed explicitly, cf. \eqref{JJJ0}, in both cases.

We first estimate $T_{2c}[h]$ for $j\not\in\{2^p, 2^{p}+1, 2^{p}+2\}$, To this end, we let $F_j\in W^1_\infty(\R)$ be the function defined by
\begin{equation*}
 F_{j}=f \quad   \text{on $[a_j^p,b_j^p]$}, \qquad F_{j}'=f'(x_j^p) \quad   \text{on $\R\setminus [a_j^p,b_j^p]$.} 
\end{equation*}
Then, $\|F_j'\|_{\infty}\leq \|f'\|_{\infty}$, and since $(\supp \pi_j^p)\cap[a_j^p- {\pi}/{2^p}, b_j^p+{\pi}/{2^p}]\subset [a_j^p, b_j^p]$, we conclude that  
\begin{align*}
 T_{2c}[h]  &={\bf 1}_{(-\pi,\pi)}\chi_j^p \PV\int\limits_{|s|<\frac{\pi}{2^p}} \frac{\big[\delta_{[\cdot,s]}(F_j-f'(x_j^p){\rm id}_{\R})/s\big]\big[\delta_{[\cdot,s]}(F_j+f'(x_j^p){\rm id}_{\R})/s\big]}
  {\big[1+\tau^2f'^2(x_j^p)\big]\big[1+\tau^2\big( \delta_{[\cdot,s]} f/s\big)^2\big]} \frac{(\pi_j^p \varphi h)''(\cdot-s)}{s} \,ds\\[1ex]
  &=T_{2e}[h]-T_{2f}[h],
\end{align*}
where
\begin{align*}
 T_{2e}[h]  &:={\bf 1}_{(-\pi,\pi)}\chi_j^p \PV\int_{\R} \frac{\big[\delta_{[\cdot,s]}(F_j-f'(x_j^p){\rm id}_{\R})/s\big]\big[\delta_{[\cdot,s]}(F_j+f'(x_j^p){\rm id}_{\R})/s\big]}
  {\big[1+\tau^2f'^2(x_j^p)\big]\big[1+\tau^2\big( \delta_{[\cdot,s]} f/s\big)^2\big]} \frac{(\pi_j^p \varphi h)''(\cdot-s)}{s} \,ds,\\[1ex]
  T_{2f}[h]  &:={\bf 1}_{(-\pi,\pi)}\chi_j^p \int\limits_{ \frac{\pi}{2^p}<|s|<5\pi} \frac{\big[\delta_{[\cdot,s]}(F_j-f'(x_j^p){\rm id}_{\R})/s\big]\big[\delta_{[\cdot,s]}(F_j+f'(x_j^p){\rm id}_{\R})/s\big]}
  {\big[1+\tau^2f'^2(x_j^p)\big]\big[1+\tau^2\big( \delta_{[\cdot,s]} f/s\big)^2\big]} \frac{(\pi_j^p \varphi h)''(\cdot-s)}{s} \,ds.
\end{align*}
Integrating  the second term by parts, it follows from Lemma \ref{L:A1} $(i)$ that  
\begin{equation}\label{E8}
\begin{aligned}
 \|T_{2b}[h]\|_{L_2((-\pi,\pi))}&\leq\|T_{2e}[h]\|_{L_2((-\pi,\pi))}+\|T_{2f}[h]\|_{L_2((-\pi,\pi))}+ K\| h\|_{H^s(\s)}\\[1ex]
 &\leq C\|F_j'-f'(x_j^p)\|_{\infty}\|\pi_j^p h\|_{H^2(\s)}+ K\| h\|_{H^s(\s)} \\[1ex]
 &\leq C\|f'-f'(x_j^p)\|_{L_\infty((a_j^p,b_j^p))}\|\pi_j^p h\|_{H^2(\s)}+ K\| h\|_{H^s(\s)} \\[1ex]
&\leq \frac{\mu}{18}\|\pi_j^p h\|_{H^2(\s)}+ K\| h\|_{H^s(\s)},
\end{aligned}
\end{equation}
if $p$ is sufficiently large.

Noticing that \eqref{p4b} implies
\begin{align*}
 \|(A_{j,\tau}[\pi_j^ph])'-((A_{j,\tau}^0 -A_{j,\tau}^1 -A_{j,\tau}^2)[\pi_j^p h])' \|_{L_2(\s)}\leq \|\Phi_b(0)[\pi_j^p h]\|_{H^1(\s)}\leq K\|h\|_{H^s(\s)},
\end{align*}
for  $j\not\in\{2^p, 2^{p}+1, 2^{p}+2\}$ the desired conclusion \eqref{DE3} follows from \eqref{E1}-\eqref{E3} and \eqref{Chopin} -\eqref{E8}.

In the second case, when $j\in\{2^p, 2^{p}+1, 2^{p}+2\}$, we decompose 
\[
T_{2c}[h]={\bf 1}_{(-\pi,a_j^p]}T_{2c}[h]+{\bf 1}_{[b_j^p,\pi)}T_{2c}[h],
\]
and we estimate the new two terms   separately,  by using  similar  arguments to those that lead to \eqref{E8}. 
For example, when considering ${\bf 1}_{[b_j^p,\pi)}T_{2c}[h]$, the function $F_j\in W^1_\infty(\R)$ is  defined by
\begin{equation*}
 F_{j}=f \quad   \text{on $[b_j^p,\pi+\pi/2^p]$}, \qquad F_{j}'=f'(x_j^p) \quad   \text{on $\R\setminus [b_j^p,\pi+\pi/2^p]$.} 
\end{equation*} 
The other details are analogous  to the ones used above, and therefore we omit them.
This completes our argument.
\end{proof}

We are now in the position to prove the desired estimates \eqref{13K} and \eqref{14K}, and therewith the desired generator property.
\begin{thm}\label{T:1} Given $f\in H^s(\s)$, $s\in(3/2,2)$ it holds that
\begin{align*} 
  -\Phi(f)\in \kH(H^2(\s), H^1(\s)).
 \end{align*}
\end{thm}
\begin{proof}
 Letting $a_\tau $ and $b_\tau $ denote the functions defined in \eqref{at}, it follows from \eqref{r0}, \eqref{MP2}, \eqref{p4a}, and \eqref{UET1} that there exists a constant 
 $\vartheta\geq 2\pi$ such that
 \begin{align*}
  \sup_{\tau\in[0,1]}\|a_\tau\|_{\infty}\leq \vartheta\qquad\text{and}\qquad \frac{1}{\vartheta}\leq \min_{\tau\in[0,1]} b_\tau\leq \max_{\tau\in[0,1]}b_\tau\leq \vartheta.
 \end{align*}
Given $\tau\in[0,1]$, $3\leq p\in\N$, a finite $p$-partition of unity  $\{\pi_j^p\,:\, 1\leq j\leq 2^{p+1}\} $, and $j\in\{1,\ldots, 2^{p+1}\}$,  
the operators $A_{j,\tau}$ defined in \eqref{FM1} are elements of the set $\{\bA_{a,b}\,:\, |a|\leq \vartheta, \vartheta^{-1}\leq b\leq \vartheta\},$
where $\bA_{a,b}$ denotes the Fourier multiplier
\[
\bA_{a,b}:=a\p_x-b(-\p_x^2)^{1/2}.
\]
A Fourier series expansion argument shows that there exists a constant $ \kappa_0\geq1$ such that
 \begin{align}\label{13Ka}
&\lambda- \bA_{a,b}\in{\rm Isom}(H^2(\s), H^1(\s)),\\[1ex]
\label{14Ka}
& \kappa_0\|(\lambda-\bA_{a,b})[h]\|_{H^1(\s)}\geq  |\lambda|\cdot\|h\|_{H^1(\s)}+\|h\|_{H^2(\s)} 
\end{align}
for all $|a|\leq \vartheta,$ $\vartheta^{-1}\leq b\leq \vartheta$, $\re\lambda\geq1$, and $h\in H^2(\s)$.

We now chose $\mu=1/2\kappa_0 $ and infer from Theorem \ref{T1} that there exist  $p\geq3$, a finite $p$-partition of unity  $\{\pi_j^p\,:\, 1\leq j\leq 2^{p+1}\} $, a constant $K=K(p)$,
and for each  $ p\in\{1,\ldots,2^{p+1}\}$ and $\tau\in[0,1]$ there 
exist   operators $\bA_{j,\tau}\in\kL(H^2(\s), H^{1}(\s))$ such that 
 \begin{equation*} 
  \|\pi_j^p\Phi(\tau f)[h]-\bA_{j,\tau}[\pi^p_j h]\|_{H^1(\s)}\leq \frac{1}{2\kappa_0} \|\pi_j^p h\|_{H^2(\s)}+K\|  h\|_{H^{s'}(\s)}
 \end{equation*}
 for all $ j\in\{1,\ldots, 2^{p+1}\}$, $\tau\in[0,1],$ and  $h\in H^2(\s)$. 
The latter estimate together with  \eqref{14Ka} yields
\begin{align*}
   \kappa_0\|\pi_j^p(\lambda-\Phi(\tau f))[h]\|_{H^1(\s)}&\geq \kappa_0\|(\lambda-\bA_{j,\tau})[\pi^p_jh]\|_{H^1(\s)}-\kappa_0\|\pi_j^p\Phi(\tau f)[h]-\bA_{j,\tau}[\pi^p_j h]\|_{H^1(\s)}\\[1ex]
   &\geq |\lambda|\cdot\|\pi^p_jh\|_{H^1(\s)}+ \frac{1}{2}\|\pi^p_j h\|_{H^2(\s)}-\kappa_0K\|  h\|_{H^{s'}(\s)}.
 \end{align*}
The Remark \ref{R:1} together with Young's inequality and \eqref{IP} ensures now the existence of constants $\kappa\geq1$ and $\omega\geq 1$ such that 
\begin{align*}
& \kappa\|(\lambda-\Phi(\tau f))[h]\|_{H^1(\s)}\geq  |\lambda|\cdot\|h\|_{H^1(\s)}+\|h\|_{H^2(\s)}
\end{align*}
for all $\tau\in[0,1]$, $\lambda\in\C$ with $\re \lambda\geq \omega$,  and $h\in H^2(\s)$. Choosing $\tau=1$, we have established  \eqref{14K}.
Moreover, for the particular choice $\lambda=\omega$ in the latter inequality, the method of continuity, see e.g. \cite[Proposition 1.1.1]{Am95}, together with \eqref{13Ka} implies, in view of $\Phi(0)=\bA_{0,2\pi},$ that indeed
$\omega-\Phi(f)\in{\rm Isom}(H^2(\s), H^1(\s))$. This completes the proof.
\end{proof}

We now come to the proof of our first main result.

\begin{proof}[Proof of Theorem \ref{MT1}] Let $s\in(3/2,2)$ be fixed and set $\E_1:=H^2(\s)$, $\E_0:=H^1(\s)$, $ \alpha:=s-1,$ and $\beta:=\ov s-1$, where we fix  some $\ov s\in (3/2,2)$.
 In view of the  well-known interpolation property
\begin{align}\label{IP}
[H^{s_0}(\X),H^{s_1}(\X)]_\theta=H^{(1-\theta)s_0+\theta s_1}(\X),\qquad\theta\in(0,1),\, 0\leq s_0\leq s_1<\infty,\, \X\in\{\s,\R\},
\end{align}
 it holds
\[
\E_\alpha=[\E_0,\E_1]_\alpha=H^{ s}(\s) \qquad\text{and}\qquad \E_\beta=[\E_0,\E_1]_\beta=H^{\ov s}(\s).
\]
The relation \eqref{Gom} together with  Theorem \ref{T:1} ensures that the assumptions of Theorem \ref{T:A} are fulfilled by the Muskat problem \eqref{AF}, and we conclude that \eqref{AF} has, for each $f\in H^s(\s)$,
a maximal classical solution  $f:=f(\cdot; f_0)$, with
\begin{align*}
  &f\in C([0,T_+(f_0)),H^s(\s))\cap C((0,T_+(f_0)), H^2(\s))\cap C^1((0,T_+(f_0)), H^1(\s)), \quad\text{ $T_+(f_0)\leq\infty$,}\\[1ex] 
  &f\in  C^{s-\ov s}([0,T], H^{\ov s}(\s))\quad\text{for all  $T\in(0,T_+(f_0))$.}
\end{align*}

Concerning the uniqueness claim in Theorem \ref{MT1}, it suffices to prove that any  function
\begin{equation*} 
  f\in C([0,T],H^s(\s))\cap C((0,T], H^2(\s))\cap C^1((0,T]), H^1(\s)),\qquad T>0,
 \end{equation*}  
 that solves \eqref{AF} pointwise,  satisfies 
  \begin{equation} \label{T:EEE}
   f\in C^\eta([0,T],H^{\ov s}(\s)),
 \end{equation}
where $\eta:=(s-\ov s)/s\in(0,s-\ov s).$
To this end, we note that   the assumption  $f\in C([0,T],H^s(\s))$ together with \eqref{PHI0}-\eqref{PHI1} and Lemma \ref{L:2} yields that  $\p_tf\in {\rm BC}((0,T], L_2(\s)),$
and the property \eqref{T:EEE}, and implicitly the claim $(i)$ of the theorem,  follows from
\begin{align*}
  \| f(t)-f(s)\|_{H^{\ov s}(\s)}\leq  \|f(t)- f(s)\|_{L_2(\s)}^{1-\ov s/s}\| f(t)- f(s)\|_{H^s(\s)}^{\ov s/s}\leq C|t-s|^{\eta},\qquad t,s\in[0,T].
 \end{align*}
 
 The criterion for global existence stated at $(ii)$ follows directly from Theorem \ref{T:A}.
Finally, the real-analyticity property   $(iii)$  is obtained by using \eqref{Gom} and a parameter trick which appears, in other forms, also in  \cite{An90, ES96, PSS15}.
Since the proof of this claim is almost identical to that in the nonperiodic case \cite[Theorem 1.3]{M16x} we omit here the details. 
\end{proof}

\section{Stability analysis}\label{S5}

We now study the stability properties  of the stationary solution $f\equiv0$ to the evolution problem  \eqref{P}.
To this end we first note that for each constant $c\in\R$, the function $f\equiv c$  defines  a stationary solution to \eqref{P}.
Therefore, it is suitable to introduce the following Hilbert spaces
\[
H^r_0(\s):=\{f\in H^r(\s)\,:\, \langle f\rangle=0\}, \qquad r\geq0,
\]
where 
\[
 \langle f\rangle:=\frac{1}{2\pi}\int_{-\pi}^\pi f(x)\, dx
\]
 denotes the integral mean of $f$ over one period.
 Since the value of the  quotient $k\Delta_\rho/4\pi\mu$  influences the rate of convergence  of the solutions that start in $H^2_0(\s)$ and are initially small towards the zero solution, see Theorem \ref{MT2}, 
 we do not scale out this constant from the problem, as we did when proving Theorem \ref{MT1}.
\begin{lemma}\label{L:4}
Given $f\in H^2_0(\s)$, let 
\begin{align}\label{Psi}
 \Psi(f):=\frac{k\Delta_\rho}{4\pi\mu}\Phi(f)[f].
\end{align}
It then holds
\begin{align}\label{RPsi}
 \Psi\in {\rm C}^\infty(H^2_0(\s), H^1_0(\s)).
\end{align}
\end{lemma}
\begin{proof}
 Having established that $\langle \Psi(f)\rangle=0$ for all $f\in H^2_0(\s)$, the desired claim \eqref{RPsi} follows from \eqref{Gom}. 
 Given $f\in H^2_0(\s)$, let $$v_+=\wt v\big|_{[y>f(x)]}\in {\rm C}([y\geq f(x)])\cap {\rm C}^1([y>f(x)]),$$
 where $\wt v$ is introduced in \eqref{VF1}  for $\ov\omega$ defined in \eqref{ovom}.
 The proof of Proposition \ref{PE} shows that 
 \[\Psi(f)=\langle v_+|_{[y=f(x)]}|(-f',1)\rangle.\]
 Using Stokes's theorem, we find in view of  $\eqref{eq:S1}_1$ and of \eqref{eq:S3} that
 \begin{align*}
  \int_{-\pi}^\pi \langle v_+|_{[y=f(x)]}|(-f',1)\rangle\, dx=&\int\limits_{[y=f(x)]} \langle v_+|\nu\rangle\, d\sigma=-\int\limits_{[y>f(x)]} {\rm div}\,  v_+\, d(x,y)=0,
 \end{align*}
where we integrate over one period of the curve $[y=f(x)]$, respectively over one period of the periodic domain $[y>f(x)]$.
This proves the claim.
\end{proof}

Lemma \ref{L:4} shows in particular that the integral mean of the solutions to \eqref{P} found in Theorem \ref{MT1} is given by the integral mean of the initial data, 
and this aspect should be taken into account when studying the asymptotic stability of the zero solution.
To do so, we observe that for initial data with zero integral mean,  \eqref{P}   can be  recast as
\begin{align}\label{GP}
 \p_tf = \Psi(f),\quad t>0,\qquad f(0)=f_0,
\end{align}
 and our task reduces to studying the asymptotic stability of zero solution to \eqref{GP}.
 In this way we have eliminated the zero eigenvalue from the spectrum of  the Fr\'echet derivative $\p\Psi{0},$ as the next result shows.

\begin{lemma}\label{L:S}
 The Fr\'echet derivative $\p\Psi(0)\in\kL(H^2_0(\s), H^1_0(\s))$ is given by
 \[
 \p\Psi(0)[f]=\frac{k\Delta_\rho}{4\pi\mu}\Phi(0)[f]  \qquad\text{for $f\in H^2_0(\s)$,}
 \]
 and 
 \begin{align}\label{Spe}
  \sigma(\p\Psi(0))=\Big\{-\frac{k\Delta_\rho m}{2\mu}\,:\, 1\leq m\in\N \Big\}.
 \end{align}
\end{lemma}
\begin{proof}
 The first claim is a direct consequence of \eqref{Gom} and \eqref{Psi}. The relation \eqref{Spe} follows from the previous observation that $\Phi(0)=-2\pi(-\p_x^2)^{1/2}.$
\end{proof}

As a straight forward consequence of Lemma \ref{L:S} we obtain
\begin{align}\label{GPF}
 -\p\Psi(0)\in\kH(H^2_0(\s), H^1_0(\s)).
\end{align}

 The relation \eqref{GPF} together with the Lemmas \ref{L:4} and \ref{L:S}  are the ingredients  needed in order to  apply the principle of linearized stability \cite[Theorem 9.1.2]{L95} in the context of \eqref{GP}.
 
 \begin{proof}[Proof of Theorem \ref{MT2}]
  Noticing that
  \[
  \sup\{\re\lambda\,:\,\lambda\in\sigma(\p\Psi(0))\}=-\frac{k\Delta_\rho}{2\mu}<0,
  \]
  the assertion of Theorem \ref{MT2}  is a straightforward consequence of  \cite[Theorem 9.1.2]{L95} and of Theorem~\ref{MT1}.   
 \end{proof}
\appendix
\section{Some technical results}\label{S:A}

We collect in Lemma \ref{L:A1} some    results on the boundedness of two families of multilinear singular integral operators  which are needed in the analysis and which are extensions
of some results established in the nonperiodic case in \cite{M16x}.

  \begin{lemma}\label{L:A1}\footnote{By convention  the empty product is  equal to $1$.}
  \begin{itemize}
  \item[$(i)$] Given $n,\, m\in\N$ and  functions $a_1,\ldots, a_{n},\, b_1, \ldots,b_m:\R\to\R$ which are  Lipschitz continuous, the singular integral operator 
  $B_{m,n}(a_1,\ldots, a_{n})[b_1,\ldots,b_m,\,\cdot\,]$ defined by
\[
B_{m,n}(a_1,\ldots, a_{n})[b_1,\ldots,b_m,h](x):=\PV\int_\R  \frac{h(x-y)}{y}\cfrac{\prod_{i=1}^{m}\big(\delta_{[x,y]} b_i /y\big)}{\prod_{i=1}^{n}\big[1+\big(\delta_{[x,y]}  a_i /y\big)^2\big]}\, dy ,
\]
belongs to $\kL(L_2(\R))$ and $\|B_{m,n}(a_1,\ldots, a_{n})[b_1,\ldots,b_m,\,\cdot\,]\|_{\kL(L_2(\R))}\leq C\prod_{i=1}^{m} \|b_i'\|_{\infty}$, where $C$ is a constant depending only 
on $n,\, m$ and $\max_{i=1,\ldots, n}\|a_i'\|_{\infty}.$\\
 \item[$(ii)$] Given $n, m \in\N$,  $r\in(3/2,2)$, $a_1,  \ldots, a_{n},\, b_1,\ldots,b_{m}\in H^r(\s)$, and $h\in L_2(\R)$ we define
\[
 A_{m,n}(a_1, \ldots, a_{n})[b_1, \ldots,b_{m}, h](x):=\PV\int_\R\frac{\prod_{i=1}^{m}\big(\delta_{[x,y]} b_i/y\big )}{\prod_{i=1}^{n}\big[1+\big(\delta_{[x,y]} a_i /y\big)^2\big]}\frac{ \delta_{[x,y]}h}{y} \, dy.
\]
Then:\\[-2ex]
\begin{itemize}
 \item[$(ii1)$] There exists a constant $C$, depending only on $r,$ $n,$ $m,$ and $\max_{i=1,\ldots, n}\|a_i\|_{H^r(\s)}$, such that  
\begin{equation*}
\big\|A_{m,n}(a_1, \ldots, a_{n})[b_1, \ldots,b_{m}, h]\big\|_{L_2(\R)}\leq C\|h\|_{L_2(\R)}\prod_{i=1}^m\|b_i\|_{H^r(\s)}
\end{equation*}
for all $  b_1,\ldots,b_m\in H^r(\s) $ and  $h\in L_2(\R)$. \\[-2ex]
\item[$(ii2)$] $A_{m,n}\in {\rm C}^{1-}((H^r(\s))^{n},\kL_{m+1}((H^r(\s))^m\times L_2(\R), L_2(\R))$.\\
\end{itemize}
\item[$(iii)$] Let $n\in\N$, $1\leq m\in  \N$, $r\in(3/2,2)$, $\tau\in(5/2-r,1)$, and $a_1,  \ldots, a_{n}\in H^r(\s)$ be given.
Then:\\[-2ex]
\begin{itemize}
 \item[$(iii1)$] There exists a constant $C$, depending only on $r $  and $\tau$    such that  
\begin{equation*}
 \big\|A_{m,n}(a_1, \ldots, a_{n})[b_1, \ldots,b_{m}, h]\big\|_{L_2(\R)}\leq C\|h\|_{H^\tau(\R)}\|b_m\|_{H^{r-1}(\s)}\prod_{i=1}^{m-1}\|b_i'\|_{\infty}
\end{equation*}
for all $ b_1,\ldots,b_{m}\in H^r(\s) $ and all $h\in H^1(\R)$. In particular, $A_{m,n}(a_1, \ldots, a_{n})$ extends to a bounded operator 
$$A_{m,n}(a_1, \ldots, a_{n})\in\kL_{m+1}((H^r(\s))^{m-1}\times H^{r-1}(\s)\times H^\tau(\R), L_2(\R)).$$
\item[$(iii2)$] $A_{m,n}\in {\rm C}^{1-}((H^r(\s))^{n},\kL_{m+1}((H^r(\s))^{m-1}\times H^{r-1}(\s)\times H^\tau(\R), L_2(\R))$.
\end{itemize}
\end{itemize}
\end{lemma}
\begin{proof}
 The proof is identical to that of the Lemmas 3.1 and 3.4 (see also Remark 3.3) in \cite{M16x}.
\end{proof}

The following  result is used in the proof of Lemma \ref{L:3}.
 \begin{lemma}\label{L:A2}
 Let $f\in {\rm C}(\s)$ and assume that $g:=f|_{(-\pi,\pi)}\in  H^1((-\pi,\pi))$.
Then $f\in H^1(\s)$ and $f'$ is the $2\pi$-periodic extension of $g'$.
\end{lemma}
 \begin{proof}
  Let $F\in L_2(\s)$ be the $2\pi$-periodic extension of $g'$. Given $\varphi\in {\rm C}^\infty(\s)$, it holds, in virtue of $g(\pm\pi)=f(\pm\pi)=f(\pi),$ that
  \begin{align*}
   \int_\s f\varphi'\, dx=\int_{-\pi}^\pi g\varphi'\, dx=(g\varphi)(\pi)-(g\varphi)(-\pi)-\int_{-\pi}^\pi g'\varphi\, dx=-\int_\s F\varphi\, dx.
  \end{align*}
This proves the claim.
 \end{proof}

The following technical result is employed in the proof of Theorem \ref{T1}. 
\begin{lemma}\label{L:A3}
 Given $f\in H^s(\s),$ $ s\in(3/2,2) $, let 
 $\phi_6(f):\R\to\R$ be defined by
  \begin{align*}
\phi_6(f)(x) :=&\PV\int_{-\pi}^\pi \frac{1}{s}\frac{1}{1+(\delta_{[x,s]}f/s )^2}\, ds.
\end{align*}
Let further $\alpha:=s/2-3/4\in(0,1).$
It then holds $\phi_6\in{\rm C}^{\alpha}(\s)$ and 
\begin{align}\label{UET1}
 \sup_{\tau\in[0,1]}\|\phi_6(\tau f) \|_{{\rm C}^{\alpha}}<\infty.
 \end{align}
\end{lemma}
\begin{proof}
 The proof is similar to that of \cite[Lemma A.1]{M16x}.
\end{proof}

In the proof of Theorem \ref{MT1} we use  results for abstract quasilinear parabolic problem obtained by H. Amann in a more general context, see \cite[Section 12]{Am93} and  \cite{Am86, Am88}, and 
which we collect in Theorem \ref{T:A}.

\begin{thm}\label{T:A}
Let $\E_0, \E_1$ be Banach spaces with compact dense embedding $\E_1\hookrightarrow \E_0$, let $[\cdot,\cdot]_\theta$ denote the complex interpolation functor, and let $\E_\theta:=[\E_0,\E_1]_\theta$ for all $0<\theta<1.$
Let further $0<\beta<\alpha<1$ and assume that  
\[-\Phi\in C^{1-}(\E_\beta, \mathcal{H}(\E_1,\E_0)).\]
The following assertions hold for the quasilinear evolution problem 
\begin{equation}
  \p_t f=\Phi(f)[f],\quad t>0,\qquad f(0)=f_0.\tag{QP}\\[-0.1ex]
\end{equation}
{\bf\em  Existence:} Given  $f_0\in \E_\alpha,$ the problem {\rm (QP)}
possesses a   maximal solution  
\[ f:=f(\cdot; f_0)\in C([0,T_+(f_0)),\E_\alpha)\cap C((0,T_+(f_0)), \E_1)\cap C^1((0,T_+(f_0)), \E_0) \cap C^{\alpha-\beta}([0,T], \E_\beta)\]
for all $T\in(0,T_+(f_0))$, with $T_+(f_0)\in(0,\infty]$. \medskip

\noindent{\bf\em  Uniqueness:} If $\wt T\in(0,\infty]$, $\eta\in(0,\alpha-\beta]$,  and $\wt f \in  C((0,\wt T), \E_1)\cap C^1((0,\wt T), \E_0)$ satisfies 
\[
\wt f\in  C^{\eta}([0,T], \E_\beta) \qquad\text{for all $T\in(0,\wt T)$}
\]
and solves {\rm (QP)}, then $\wt T\leq T_+(f_0)$ and $\wt f=f$ on $[0,\wt T)$.\medskip

\noindent{\bf\em  Criterion for global existence:} If $f:[0,T]\cap[0,T_+(f_0))\to \E_\alpha$ is bounded for all $T>0$, then  
\[
T_+(f_0)=\infty.
\]
\noindent{\bf\em  Continuous dependence of initial data:}
The  mapping $[(t,f_0)\mapsto f(t;f_0)]$ defines a  semiflow on $\E_\alpha $ and, if   $\Phi\in C^{\omega}(\E_\beta, \mathcal{L}(\E_1,\E_0))$, then
\[[(t,f_0)\mapsto f(t;f_0)]:\{(t,f_0)\,:\, f_0\in\E_\alpha, t\in(0,T_+(f_0))\}\to \E_\alpha\]
is a real-analytic map too.
\end{thm}


\begin{thebibliography}{10}

\bibitem{Am86}
H.~Amann.
\newblock {Quasilinear parabolic systems under nonlinear boundary conditions}.
\newblock {\em Arch. Rational Mech. Anal.}, 92(2):153--192, 1986.

\bibitem{Am88}
H.~Amann.
\newblock {Dynamic theory of quasilinear parabolic equations. {I}. {A}bstract
  evolution equations}.
\newblock {\em Nonlinear Anal.}, 12(9):895--919, 1988.

\bibitem{Am93}
H.~Amann.
\newblock {Nonhomogeneous linear and quasilinear elliptic and parabolic
  boundary value problems}.
\newblock In {\em {Function spaces, differential operators and nonlinear
  analysis ({F}riedrichroda, 1992)}}, volume 133 of {\em {Teubner-Texte
  Math.}}, pages 9--126. Teubner, Stuttgart, 1993.

\bibitem{Am95}
H.~Amann.
\newblock {\em {Linear and Quasilinear Parabolic Problems. {V}ol. {I}}},
  volume~89 of {\em {Monographs in Mathematics}}.
\newblock Birkh\"{a}user Boston, Inc., Boston, MA, 1995.
\newblock Abstract linear theory.

\bibitem{A04}
D.~M. Ambrose.
\newblock {Well-posedness of two-phase {H}ele-{S}haw flow without surface
  tension}.
\newblock {\em European J. Appl. Math.}, 15(5):597--607, 2004.

\bibitem{A14}
D.~M. Ambrose.
\newblock {The zero surface tension limit of two-dimensional interfacial
  {D}arcy flow}.
\newblock {\em J. Math. Fluid Mech.}, 16(1):105--143, 2014.

\bibitem{An90}
S.~B. Angenent.
\newblock {Nonlinear analytic semiflows}.
\newblock {\em Proc. Roy. Soc. Edinburgh Sect. A}, 115(1-2):91--107, 1990.

\bibitem{BV14}
B.~V. Bazaliy and N.~Vasylyeva.
\newblock {The two-phase {H}ele-{S}haw problem with a nonregular initial
  interface and without surface tension}.
\newblock {\em Zh. Mat. Fiz. Anal. Geom.}, 10(1):3--43, 152, 155, 2014.

\bibitem{Be88}
J.~Bear.
\newblock {\em {Dynamics of Fluids in Porous Media}}.
\newblock Dover Publications, New York, 1988.

\bibitem{BCG14}
L.~C. Berselli, D.~C\'{o}rdoba, and R.~Granero-Belinch\'{o}n.
\newblock {Local solvability and turning for the inhomogeneous {M}uskat
  problem}.
\newblock {\em Interfaces Free Bound.}, 16(2):175--213, 2014.

\bibitem{B06}
P.~Byers.
\newblock {Existence time for the {C}amassa-{H}olm equation and the critical
  {S}obolev index}.
\newblock {\em Indiana Univ. Math. J.}, 55(3):941--954, 2006.

\bibitem{CCFG13}
A.~Castro, D.~C\'{o}rdoba, C.~Fefferman, and F.~Gancedo.
\newblock {Breakdown of smoothness for the {M}uskat problem}.
\newblock {\em Arch. Ration. Mech. Anal.}, 208(3):805--909, 2013.

\bibitem{CCFGL12}
A.~Castro, D.~C\'{o}rdoba, C.~Fefferman, F.~Gancedo, and
  M.~L\'{o}pez-Fern\'{a}ndez.
\newblock {Rayleigh-{T}aylor breakdown for the {M}uskat problem with
  applications to water waves}.
\newblock {\em Ann. of Math. (2)}, 175(2):909--948, 2012.

\bibitem{CGFL11}
A.~Castro, D.~C\'{o}rdoba, C.~L. Fefferman, F.~Gancedo, and
  M.~L\'{o}pez-Fern\'{a}ndez.
\newblock {Turning waves and breakdown for incompressible flows}.
\newblock {\em Proc. Natl. Acad. Sci. USA}, 108(12):4754--4759, 2011.

\bibitem{BCS16}
C.~H.~A. Cheng, R.~Granero-Belinch\'{o}n, and S.~Shkoller.
\newblock {Well-posedness of the {M}uskat problem with {$H^2$} initial data}.
\newblock {\em Adv. Math.}, 286:32--104, 2016.

\bibitem{CGCSS14x}
P.~Constantin, D.~C\'{o}rdoba, F.~Gancedo, L.~Rodr\'guez-Piazza, and R.~M.
  Strain.
\newblock {On the Muskat problem: Global in time results in 2D and 3D}.
\newblock {\em Amer. J. Math.}, 138:1455--1494, 2016.

\bibitem{CCGS13}
P.~Constantin, D.~C\'{o}rdoba, F.~Gancedo, and R.~M. Strain.
\newblock {On the global existence for the {M}uskat problem}.
\newblock {\em J. Eur. Math. Soc. (JEMS)}, 15(1):201--227, 2013.

\bibitem{CGSV15x}
P.~Constantin, V.~Vicol, R.~Shvydkoy, and F.~Gancedo.
\newblock { Global regularity for 2D Muskat equations with finite slope}.
\newblock {\em Ann. I. H. Poincar\'{e} -- AN}, 2016.
\newblock http://dx.doi.org/10.1016/j.anihpc.2016.09.001.

\bibitem{CCG11}
A.~C\'{o}rdoba, D.~C\'{o}rdoba, and F.~Gancedo.
\newblock {Interface evolution: the {H}ele-{S}haw and {M}uskat problems}.
\newblock {\em Ann. of Math. (2)}, 173(1):477--542, 2011.

\bibitem{CCG13b}
A.~C\'{o}rdoba, D.~C\'{o}rdoba, and F.~Gancedo.
\newblock {Porous media: the {M}uskat problem in three dimensions}.
\newblock {\em Anal. PDE}, 6(2):447--497, 2013.

\bibitem{CG07}
D.~C\'{o}rdoba and F.~Gancedo.
\newblock {Contour dynamics of incompressible 3-{D} fluids in a porous medium
  with different densities}.
\newblock {\em Comm. Math. Phys.}, 273(2):445--471, 2007.

\bibitem{CG10}
D.~C\'{o}rdoba and F.~Gancedo.
\newblock {Absence of squirt singularities for the multi-phase {M}uskat
  problem}.
\newblock {\em Comm. Math. Phys.}, 299(2):561--575, 2010.

\bibitem{CGO14}
D.~{C\'{o}rdoba Gazolaz}, R.~Granero-Belinch\'{o}n, and R.~Orive-Illera.
\newblock {The confined {M}uskat problem: differences with the deep water
  regime}.
\newblock {\em Commun. Math. Sci.}, 12(3):423--455, 2014.

\bibitem{EEM09c}
M.~Ehrnstr\"{o}m, J.~Escher, and B.-V. Matioc.
\newblock {Steady-state fingering patterns for a periodic {M}uskat problem}.
\newblock {\em Methods Appl. Anal.}, 20(1):33--46, 2013.

\bibitem{EMM12a}
J.~Escher, A.-V. Matioc, and B.-V. Matioc.
\newblock {A generalized {R}ayleigh-{T}aylor condition for the {M}uskat
  problem}.
\newblock {\em Nonlinearity}, 25:73--92, 2012.

\bibitem{EM11a}
J.~Escher and B.-V. Matioc.
\newblock {On the parabolicity of the {M}uskat problem: well-posedness,
  fingering, and stability results}.
\newblock {\em Z. Anal. Anwend.}, 30(2):193--218, 2011.

\bibitem{EMW15}
J.~Escher, B.-V. Matioc, and C.~Walker.
\newblock {The domain of parabolicity for the Muskat problem}.
\newblock {\em Indiana Univ. Math. J.}, 2016.
\newblock to appear. arXiv:1507.02601.

\bibitem{ES96}
J.~Escher and G.~Simonett.
\newblock {Analyticity of the interface in a free boundary problem}.
\newblock {\em Math. Ann.}, 305(3):439--459, 1996.

\bibitem{FT03}
A.~Friedman and Y.~Tao.
\newblock {Nonlinear stability of the {M}uskat problem with capillary pressure
  at the free boundary}.
\newblock {\em Nonlinear Anal.}, 53(1):45--80, 2003.

\bibitem{GS14}
F.~Gancedo and R.~M. Strain.
\newblock {Absence of splash singularities for surface quasi-geostrophic sharp
  fronts and the {M}uskat problem}.
\newblock {\em Proc. Natl. Acad. Sci. USA}, 111(2):635--639, 2014.

\bibitem{GG14}
J.~G\'{o}mez-Serrano and R.~Granero-Belinch\'{o}n.
\newblock {On turning waves for the inhomogeneous {M}uskat problem: a
  computer-assisted proof}.
\newblock {\em Nonlinearity}, 27(6):1471--1498, 2014.

\bibitem{GB14}
R.~Granero-Belinch\'{o}n.
\newblock {Global existence for the confined {M}uskat problem}.
\newblock {\em SIAM J. Math. Anal.}, 46(2):1651--1680, 2014.

\bibitem{BS16x}
R.~Granero-Belinch\'{o}n and S.~Shkoller.
\newblock {Well-posedness and decay to equilibrium for the Muskat problem with
  discontinuous permeability}.
\newblock 2016.
\newblock preprint. arXiv:1611.06147.

\bibitem{HTY97}
J.~Hong, Y.~Tao, and F.~Yi.
\newblock {Muskat problem with surface tension}.
\newblock {\em J. Partial Differential Equations}, 10(3):213--231, 1997.

\bibitem{JKL93}
J.~K. Lu.
\newblock {\em {Boundary Value Problems for Analytic Functions}}, volume~16 of
  {\em {Series in Pure Mathematics}}.
\newblock World Scientific Publishing Co., Inc., River Edge, NJ, 1993.

\bibitem{L91}
A.~Lunardi.
\newblock {An introduction to geometric theory of fully nonlinear parabolic
  equations}.
\newblock In {\em {Qualitative aspects and applications of nonlinear evolution
  equations ({T}rieste, 1990)}}, pages 107--131. World Sci. Publ., River Edge,
  NJ, 1991.

\bibitem{L95}
A.~Lunardi.
\newblock {\em {Analytic Semigroups and Optimal Regularity in Parabolic
  Problems}}.
\newblock {Progress in Nonlinear Differential Equations and their Applications,
  16}. Birkh\"{a}user Verlag, Basel, 1995.

\bibitem{M16x}
B.-V. Matioc.
\newblock {The Muskat problem in 2D: equivalence of formulations,
  well-posedness, and regularity results}.
\newblock 2016.
\newblock preprint. arXiv:1610.05546.

\bibitem{M17x}
B.-V. Matioc.
\newblock {Viscous displacement in porous media: the Muskat problem in 2D}.
\newblock {\em Trans. Amer. Math. Soc.}, 2017.
\newblock to appear. arXiv:1701.00992.

\bibitem{TM86}
T.~Murai.
\newblock {Boundedness of singular integral operators of {C}alder\'{o}n type.
  {VI}}.
\newblock {\em Nagoya Math. J.}, 102:127--133, 1986.

\bibitem{Mu34}
M.~Muskat.
\newblock {Two fluid systems in porous media. The encroachment of water into an
  oil sand}.
\newblock {\em Physics}, 5:250--264, 1934.

\bibitem{PSS15}
J.~Pr\"{u}ss, Y.~Shao, and G.~Simonett.
\newblock {On the regularity of the interface of a thermodynamically consistent
  two-phase {S}tefan problem with surface tension}.
\newblock {\em Interfaces Free Bound.}, 17(4):555--600, 2015.

\bibitem{PS16}
J.~Pr\"{u}ss and G.~Simonett.
\newblock {\em {Moving Interfaces and Quasilinear Parabolic Evolution
  Equations}}.
\newblock {Monographs in Mathematics}. Birkh\"{a}user Verlag, 2016.

\bibitem{PS16x}
J.~Pr\"{u}ss and G.~Simonett.
\newblock {On the Muskat flow}.
\newblock {\em Evol. Equ. Control Theory}, 5(4):631--645, 2016.

\bibitem{SCH04}
M.~Siegel, R.~E. Caflisch, and S.~Howison.
\newblock {Global existence, singular solutions, and ill-posedness for the
  {M}uskat problem}.
\newblock {\em Comm. Pure Appl. Math.}, 57(10):1374--1411, 2004.

\bibitem{T16}
S.~Tofts.
\newblock {On the existence of solutions to the Muskat poroblem with surface
  tension}.
\newblock {\em J. Math. Fluid Mech.}, pages 1--31, 2016.

\bibitem{T04}
A.~Torchinsky.
\newblock {\em {Real-variable methods in harmonic analysis}}.
\newblock Dover Publications, Inc., Mineola, NY, 2004.
\newblock Reprint of the 1986 original [Dover, New York; MR0869816].

\bibitem{Y96}
F.~Yi.
\newblock {Local classical solution of {M}uskat free boundary problem}.
\newblock {\em J. Partial Differential Equations}, 9(1):84--96, 1996.

\end{thebibliography}
\end{document}